\documentclass[12pt, two side]{article}
\usepackage{amsfonts}

\usepackage{latexsym,amsmath,amsfonts,mathrsfs,wasysym,times,lineno}
\usepackage{amsthm,amscd,amsfonts,newlfont}
\usepackage{amssymb, upref, color}
\usepackage{pst-all}
\usepackage{epsf, graphicx, subfigure, verbatim, epic}
\usepackage{latexsym}
\usepackage[colorlinks]{hyperref}
\usepackage{mathrsfs}
\usepackage{array}

\numberwithin{equation}{section}

\theoremstyle{plain}
\newtheorem{exam}{Example}[section]
\newtheorem{theorem}[exam]{Theorem}

\newtheorem{remark}[exam]{Remark}

\newtheorem{proposition}[exam]{Proposition}
\newtheorem{definition}[exam]{Definition}
\newtheorem{corollary}[exam]{Corollary}

 \def\eps{{\epsilon}}
 \def\ov{\overline}
 \def\S{{\mathbb S}}
 \def\R{{\mathbb R}}
 \def\N{{\mathbb N}}
 \def\Q{{\mathbb Q}}

\title{Finite cyclicity of some graphics through a nilpotent point of saddle type inside quadratic systems\footnote{This research was supported by NSERC of Canada}}
\author{Christiane Rousseau$^{a}$, Chunhua Shan$^{b}$   and Huaiping Zhu$^{c}$ \\  \\
 $^{a}$ Department of Mathematics and Statistics and CRM, \\ University of Montreal, Montreal,
Canada H3C 3J7\\
$^{b}$Department of Mathematical and Statistical Sciences, \\ University of Alberta, Edmonton,
Canada T6G 2G1\\
$^{c}$ Department of Mathematics and Statistics and LAMPS, \\  York University, Toronto, Canada, M3J 1P3
}

\date{}
\begin{document}
\maketitle

\noindent {\bf Abstract.} In this paper we show the finite cyclicity of the two graphics $(I_{12}^1)$ and $(I_{13}^1)$ through a triple nilpotent point of saddle type inside quadratic vector fields. These results contribute to the program launched in 1994 by Dumortier, Roussarie and Rousseau (DRR program) to show the existence of  a uniform upper bound for the number of limit cycles for planar quadratic vector fields.


\vskip .3 cm

\section{Introduction}

Hilbert's 16th problem, second part, asks for the maximum number of limit cycles, called $H(n)$, as well as the relative positions of limit cycles of a polynomial vector field $P(x,y)\frac{\partial}{\partial x} + Q(x,y) \frac{\partial}{\partial y}$ as a function of $n=\max(\mathrm{deg}(P),
\mathrm{deg}(Q))$. It is still unknown whether $H(n)$ is finite.
The DRR program started in 1994 by Dumortier, Roussarie and Rousseau (\cite{DRR94(1)}) produces a procedure to prove that
$H(2)<\infty$.  The underlying idea is a compactness argument. Indeed, polynomial vector fields can be extended to the Poincar\'e sphere $\S^2$ by adding points at infinity in all directions. The number of limit cycles of a vector field depends only  on its equivalence class under affine transformations and time rescalings. Also, limit cycles in quadratic vector fields necessarily surround a unique singular point with nondegenerate linear part, and linear vector fields can have no limit cycles. Hence, it is possible to compactify the  space of equivalence classes of quadratic vector fields with a nondegenerate singular point of anti-saddle type: this yields a compact parameter space $K$.  Limit cycles in the compact set $\S^2\times K$ accumulate on \emph{graphics}, which are unions of trajectories and singular points for a given value of the parameters. The DRR program reduces the proof that $H(2)<\infty$ to the proof that each graphic $\Gamma\subset \S^2$ surrounding a nondegenerate singular point of anti-saddle type and occurring for a parameter value $A_0\in K$ has finite cyclicity in $\S^2\times K$, i.e. can produce only a finite number of limit cycles in a neighborhood $U$ of $\Gamma$ for parameter values $A$ in a neighborhood $V$ of $A_0$. Achieving the DRR program requires proving the finite cyclicity of 121 graphics in $\S^2\times K$.  This program has stimulated the development of  highly sophisticated methods to treat problems of increasing complexity. The graphics can be grouped in large classes and the strategy is to treat one class at a time. In this paper, we prove that  the two graphics through a nilpotent point of saddle type, $(I_{12}^1)$ and $(I_{13}^1)$, that do not surround a center, have finite cyclicity.  Therefore the results from this paper will bring the number of graphics of the program for which finite cyclicity is proved to 88.

In practice, in this paper we address the following questions:
\begin{itemize}
\item[(1)] We first show that a generic graphic through a nilpotent saddle of multiplicity 3 has finite multiplicity in the case where one connection is fixed. The case of codimension 3 was already treated in \cite{ZR} and it suffices to treat the case $a=-\frac12$ corresponding to $b=0$ in the DRS normal form (\cite{DRS}).
\item[(2)] In quadratic systems, we show that the genericity condition is met for $(I_{12}^1)$. This amounts to show that the integral of the divergence along the invariant parabola is nonzero. Note that the same computation shows the finite cyclicity of $(I_{9b}^2)$ when the codimension of the point is $3$ (corresponding to $\eps_2\neq0$ in \cite{DRS}).
\item[(3)] We show that a generic graphic through a nilpotent saddle of multiplicity 3 and a saddle-node with central transition has finite multiplicity in the case where one connection is fixed. As an application, this yields the finite cyclicity of the graphic $(I_{13}^1)$ inside quadratic systems.
 \end{itemize}

\begin{figure}
\begin{center}
\subfigure[$(I_{12}^1)$]
    {\includegraphics[width=3.5cm]{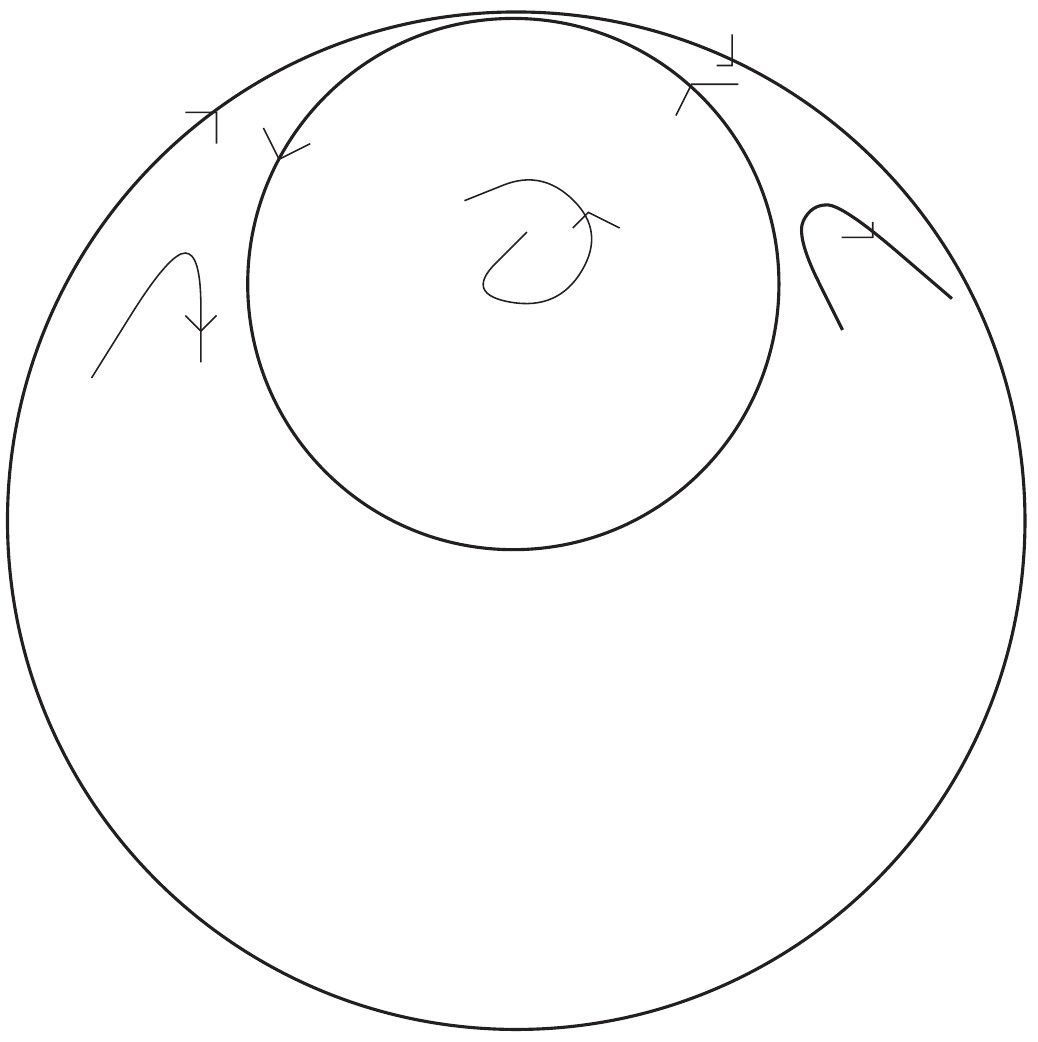}}\qquad\qquad
\subfigure[$(I_{13}^1)$]
    {\includegraphics[width=3.5cm]{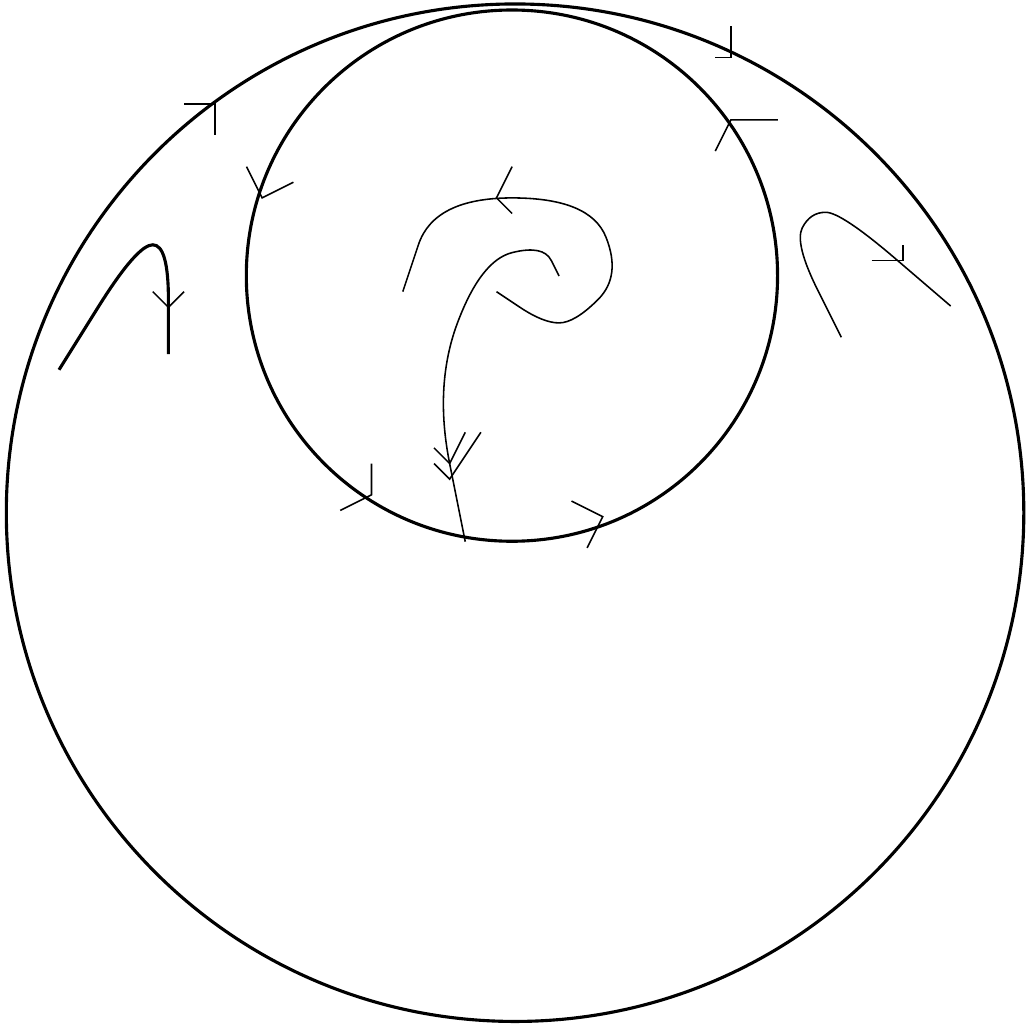}}
    \caption{Graphics for which we prove finite cyclicity}
\label{fig.list}\end{center}\end{figure}

\section{Preliminaries}

\subsection{Normal form for the unfolding of a nilpotent triple point of saddle type} We consider graphics through one singular point, which is a triple nilpotent point of saddle type. A germ of vector field in the neighborhood of such a point has the form
\begin{align}\begin{split}
\dot x&=y \\
\dot y&=x^3+ bxy +\eta x^2y + yO(x^3)+ O(y^2).\end{split}\label{normal_form_DRS}\end{align}

The unfolding of such points has been studied by Dumortier, Roussarie and Sotomayor, \cite{DRS}, including a normal form for the unfolding of the family. A different normal form has been used in \cite{ZR} for studying the finite cyclicity of generic graphics through such singular points, which is particularly suitable for applications in quadratic vector fields, where there is always an invariant line through a nilpotent point of multiplicity $3$.

\bigskip
Indeed, a germ of $C^\infty$ vector field in the neighborhood of a  nilpotent point of multiplicity $3$ of saddle type can be brought by an analytic change of coordinates to the form
\begin{align}\begin{split}
\dot x&=y+ ax^2, \\
\dot y&=y(x + \eta x^2 +o(x^2) + O(y)),\end{split}\label{normal_form_ZR}\end{align}
with $a<0$ (see Figure~\ref{fig.topotype}).
\begin{figure}[ht]\begin{center}
   \includegraphics[angle=0, width=5cm]{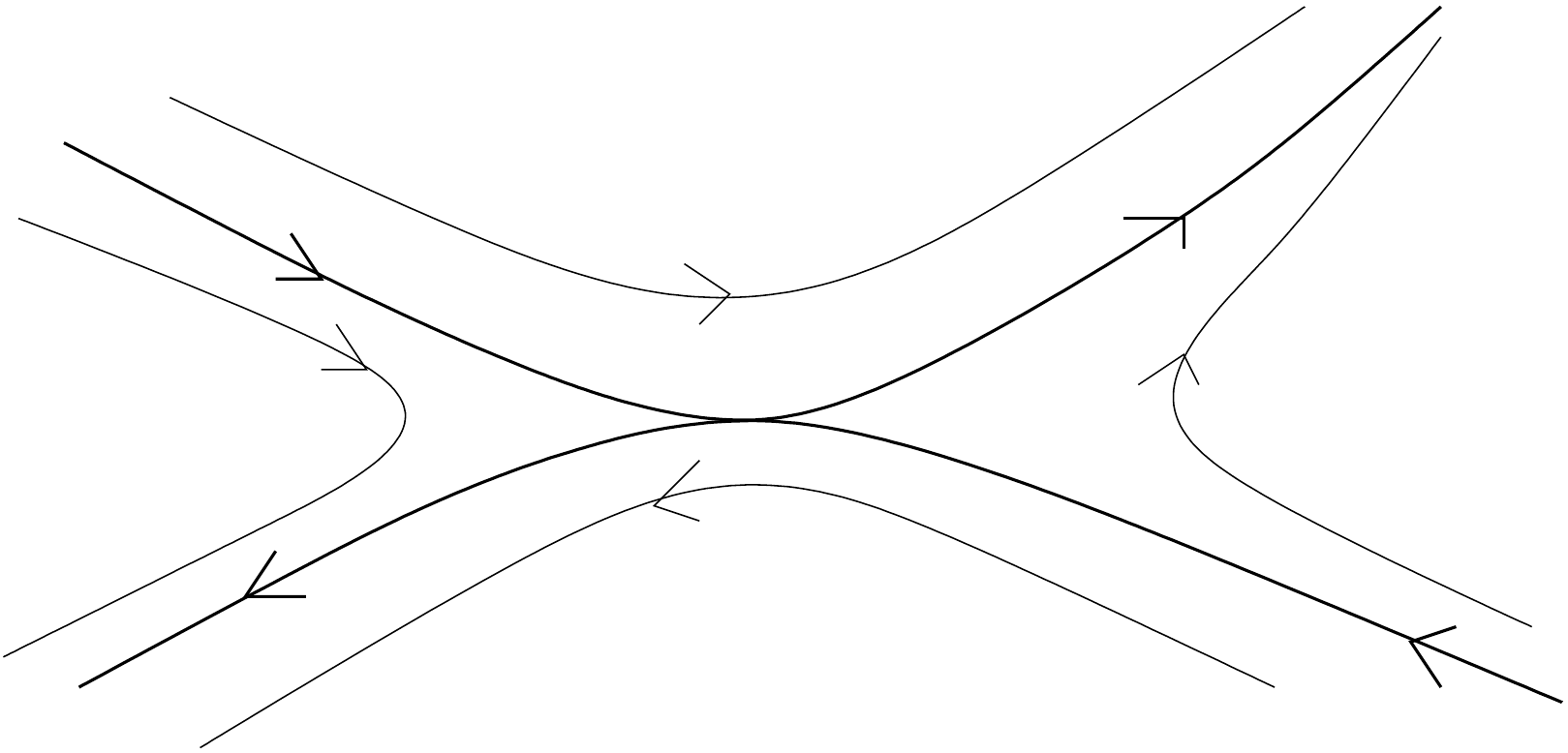}
\caption{A nilpotent saddle}
\label{fig.topotype}\end{center}\end{figure}

\bigskip

A generic unfolding depending on a multi-parameter $\lambda= (\mu_1,\mu_2,\mu_3, \mu)$ in a neighborhood of the origin has the form
\begin{align}\begin{split}
\dot x&=y+ a(\lambda)x^2 +\mu_2, \\
\dot y&=\mu_1+\mu_3 y + x^4h_1(x,\lambda)+y(x + \eta x^2 +x^3 h_2(x,\lambda))+ y^2Q(x,y,\lambda),\end{split}\label{normal_form_family}\end{align}
where $h_1(x,\lambda)= O(|\lambda|)$. Moreover, $h_1, h_2, Q$ are $C^\infty$ functions, and $Q$ can be chosen of arbitrarily high order in $\lambda$.

\subsection{Finite cyclicity of a graphic}
\begin{definition} A \emph{graphic} $\Gamma$ of a vector field $X_0$ , i.e. a union of trajectories and singular points, has \emph{finite cyclicity} inside a family $X_\lambda$ if there exists $N\in\mathbb N$, $\eps>0$ and $\delta>0$ such that any vector field $X_\lambda$ with $|\lambda|<\delta$ has at most $N$ periodic solutions  at a Hausdorff distance less than $\eps$ from  $\Gamma$. The minimum value $N$ is the \emph{cyclicity} of the graphic. \end{definition}

When studying the finite cyclicity of a graphic $\Gamma$, we need to find a uniform bound for the number of periodic solutions that can appear from it, for \emph{all} values of the multi-parameter in a small neighborhood $W$ of the origin. Typically, we need to find a uniform bound for the number of fixed points of the Poincar\'e return map or, equivalently, for the number of zeros of some displacement map between two transversal sections to the graphic. With graphics containing a nilpotent singular point there is no way to make a uniform treatment for all $\lambda\in W$, and we cover $W$ by an infinite number of sectors with conic structure, one around each direction in parameter space. On  each sector,  we give a uniform bound for the finite cyclicity. Since the set of directions in parameter space is compact, we extract a finite subcovering: the maximum of the cyclicities on each sector of the covering is the cyclicity of the graphic $\Gamma$. The method for doing this is the \emph{blow-up of the family}, which was first introduced by Roussarie.

\subsection{Blow-up of the family}
Let us make the change of parameters
\begin{equation}
(\mu_1,\mu_2,\mu_3)= (\nu^3\overline{\mu}_1,\nu^2\overline{\mu}_2,\nu\overline{\mu}_3).\label{change_parameters}\end{equation}
We take a neighborhood of the origin in parameter-space of the form $\S^2\times [0,\nu_0)\times U$, where $U$ is a neighborhood of $0$ in $\mu$-space, $\ov{M}=(\overline{\mu}_1,\overline{\mu}_2,\overline{\mu}_3)\in\S^2$ and $\nu\in [0,\nu_0)$.

\bigskip

Note that $\S^2$ is compact. Hence, to give an argument of finite cyclicity for the graphic $\Gamma$, it suffices to find  a neighborhood of each $\ov{M}=(\overline{\mu}_1,\overline{\mu}_2,\overline{\mu}_3)\in\S^2$ inside $\S^2$, a corresponding $\nu_0>0$ and a corresponding $U$ on which we can give a bound for the number of limit cycles. In our study, we will consider special values $a_0$ of $a$. It is important to note that $a(\lambda)$ depends on $\lambda$, and hence that $a-a_0$ is  a parameter in itself.

\bigskip

The way to handle this program is to do a \emph{blow-up of the family}, a technique developed by Roussarie.  For this, we introduce the weighted blow-up of the singular point $(0,0,0)$ of the three-dimensional family of vector fields obtained by adding the equation $\dot \nu=0$ to the 2-dimensional system \eqref{normal_form_family}. The blow-up transformation is given by
\begin{equation}
(x,y,\nu) = (r\ov{x}, r^2\ov{y}, r\rho),\label{blow-up_family}\end{equation}
with $r>0$ and $(\ov{x},\ov{y},\rho)\in \S^2$.
After dividing by $r$  the transformed vector field, we get a family of $C^\infty$ vector fields $\ov{X}_{A}$, depending on the parameters $A=(a- a_0,\ov{M}, \mu)$. The foliation $\{\nu=r\rho=Const\}$ is invariant under the flow. The leaves $\{r\rho=\nu\}$ with $\nu>0$ are regular two-dimensional manifolds, while the critical locus $\{r\rho=0\}$ is stratified and contains the two strata (see Figure~\ref{fig.strat}):
\begin{itemize}
\item $\S^1\times \R^+$ is the blow-up of $X_0$ (for $\lambda=0$);
\item $D_{\ov{\mu}}= \{\ov{x}^2+\ov{y}^2+\rho^2=1\mid \rho \geq0\}$.
\end{itemize}

\subsection{Limit periodic sets in the blow-up family}
The strategy for studying the finite cyclicity of $\Gamma$ is the following. We study the singular points of $\ov{X}$ on $r=\rho=0$. For $a\neq \frac12$, there will be four distinct singular points (occuring in two pairs) corresponding to $\ov{y}=0$ (for $P_1$ and $P_2$) and $\ov{y}=\frac{1-2a}{2}$ (for $P_3$ and $P_4$): see Figure~\ref{fig.strat}. Their eigenvalues appear in Table~\ref{eigenvalue}.
\begin{figure}[ht]\begin{center}
\includegraphics[angle=0,width=5.5cm]{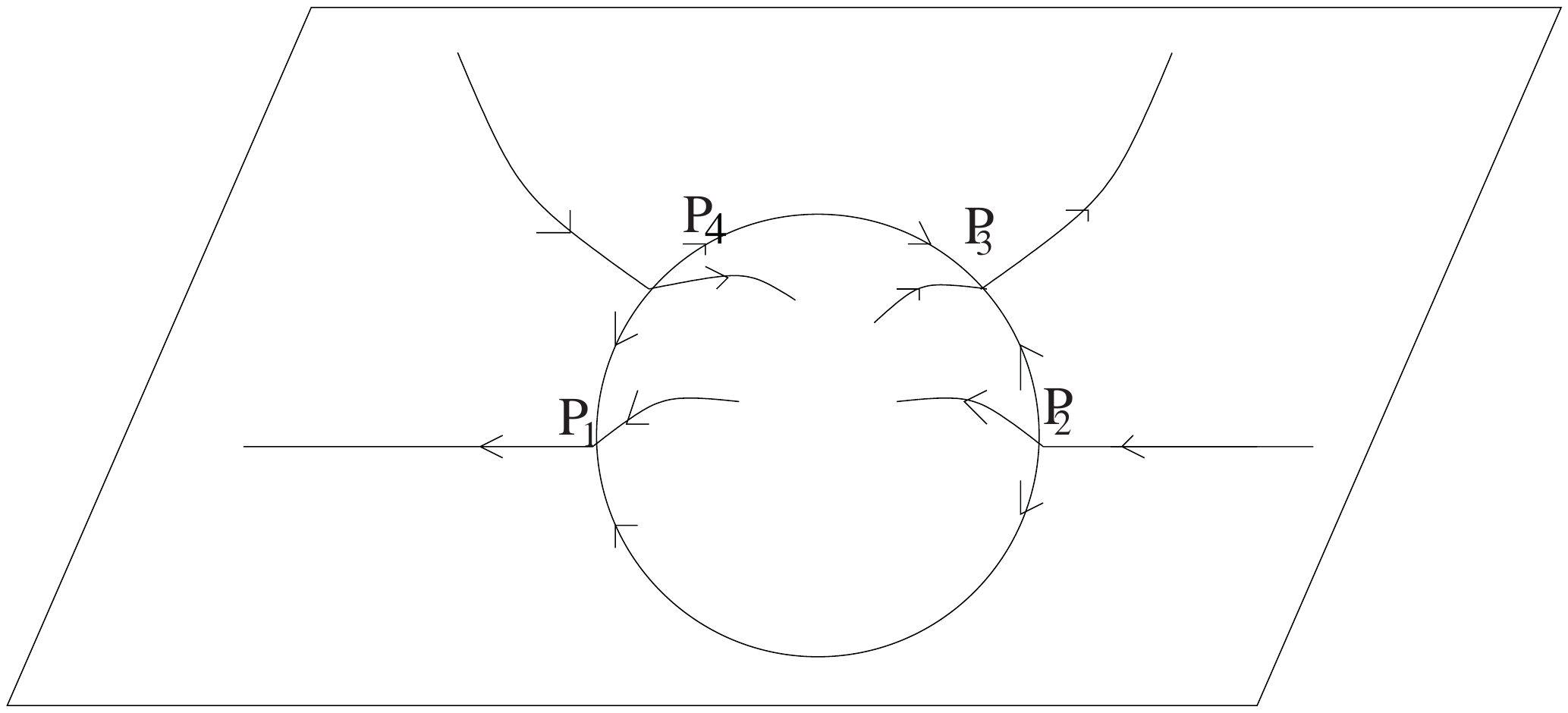}
\caption{The stratified set $\{r\rho=0\}$ in the blow-up.}
\label{fig.strat}\end{center}\end{figure}
\begin{table}[ht]\begin{center}\begin{tabular}{|c|c|c|c|}
\hline
      &  $r$       &  $\rho$      &    $ y $          \\ \hline
$P_1$ & $   -a   $ &  $\ \ a  $   &    $-(1-2a)$      \\ \hline
$P_2$ & $\ \ a   $ &  $   -a  $   &    $\ \ (1-2a)$   \\ \hline
$P_3$ & $\ \ 1/2 $ &  $ -1/2  $   &    $-(1-2a)$      \\ \hline
$P_4$ & $    -1/2$ &  $\ \ 1/2$   &    $\ \ (1-2a)$   \\ \hline
\end{tabular}\label{eigenvalue}
\caption{The eigenvalues at $P_i$ ($i=1,2,3,4$)}
\end{center}\end{table}

\bigskip

In this paper we study the finite cyclicity of a  graphic $\Gamma$ joining $P_3$ and $P_4$.
We consider a particular value $A_0=(a_0, \ov{M}_0)$. Here is the strategy for finding an upper bound for the number of limit cycles that appear for $A$ in a neighborhood of $A_0$. We determine the phase portrait of the family rescaling \eqref{family_rescaling} on $D_{\ov{\mu}}$: this allows determining \emph{limit periodic sets} $\ov{\Gamma}$, which are formed by the union of $\Gamma$ with a finite number of trajectories and singular points on $D_{\ov{\mu}}$ joining $P_4$ and $P_3$, so that their orientation will be compatible with that of $\Gamma$. The limit periodic sets to be studied appear in Table~\ref{tab.shhconvex}. They are continuous families of limit periodic sets. We use the convention to label the different types: Sxhhia, Sxhhib, etc, starting from the top. For instance, Sxhh1a corresponds to the boundary upper limit periodic set, Sxhh1b corresponds to any of the intermediate limit periodic set, and Sxhh1c corresponds to the lower periodic set through the saddle point. They come from studying the phase portrait of the \emph{family rescaling}
\begin{align} \begin{split}
\dot{\ov{x}}&= \ov{y}+a\ov{x}^2+\ov{\mu}_2,\\
\dot{\ov{y}}&=\ov{\mu}_1+\ov{\mu}_3\ov{y} +\ov{x}\ov{y},\end{split}\label{family_rescaling}
\end{align}
obtained by putting $\rho=1$ and $r=0$. It then suffices to show that each limit periodic set has finite cyclicity, i.e. to show the existence of an upper bound for the number of periodic solutions of $\ov{X}_A$ for $A$ in a small neighborhood of $A_0$.
\begin{table}[ht]
\begin{center}\begin{tabular}{|c|c|c|}
\hline
   \includegraphics[angle=0, width=3.2cm]{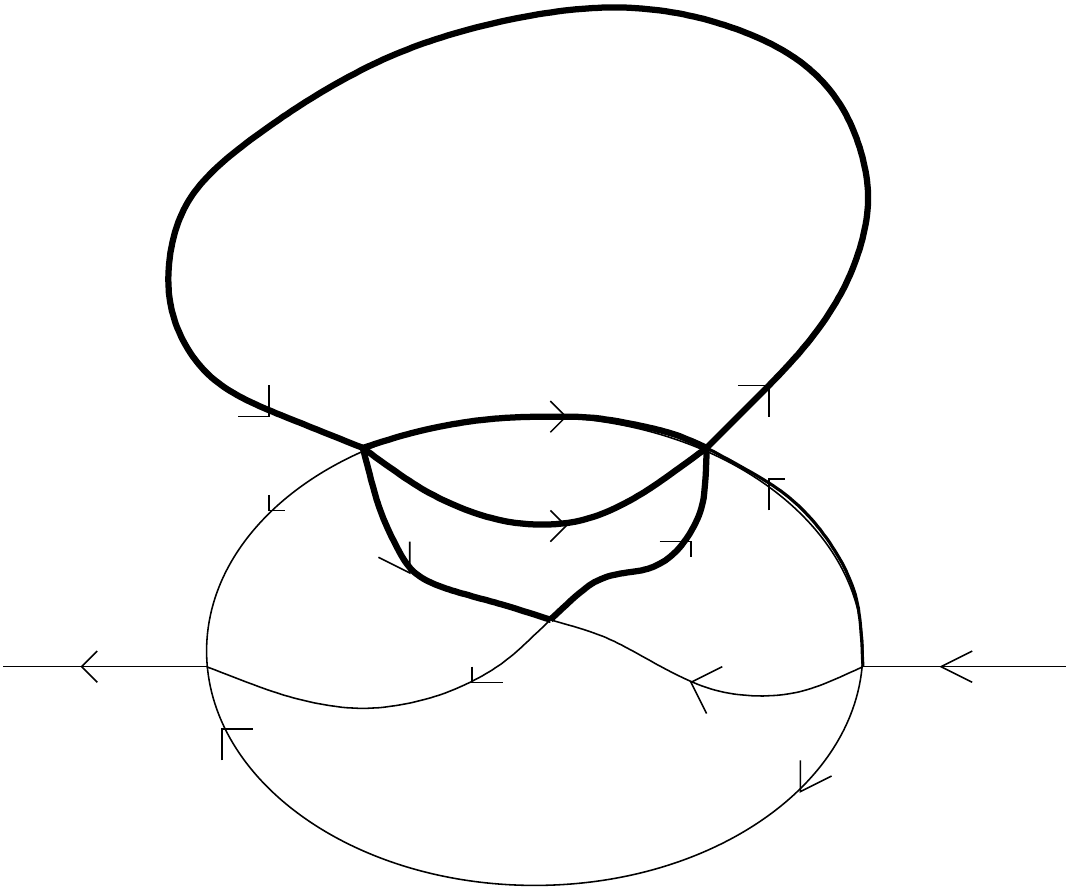}
&  \includegraphics[angle=0,width=3.2cm]{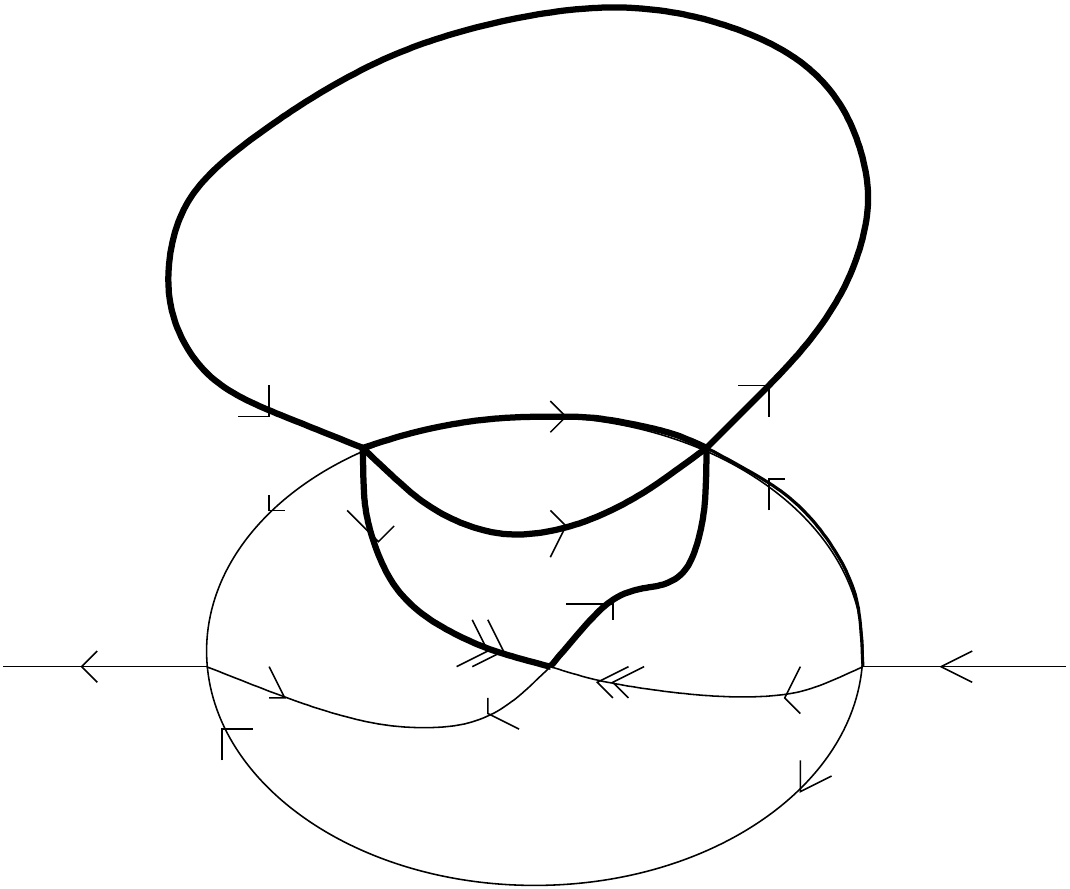}
&  \includegraphics[angle=0,width=3.2cm]{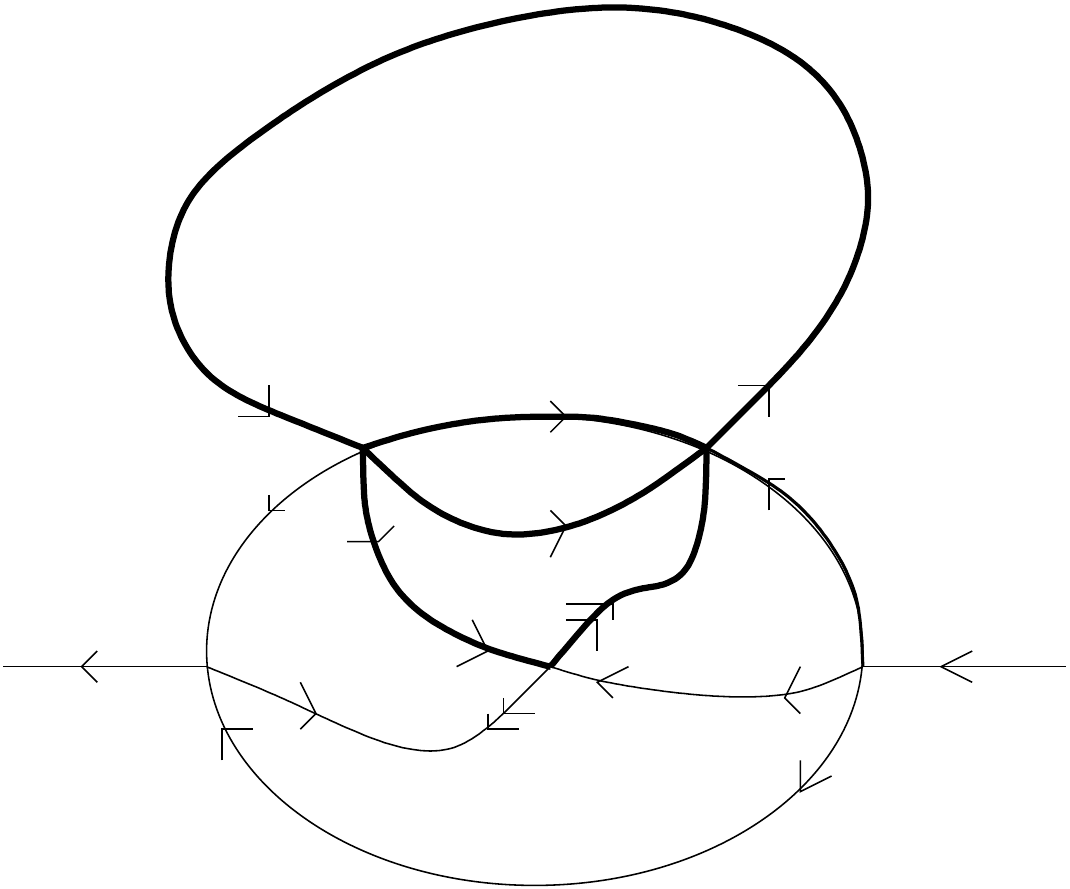}
\\ \hline   Sxhh1  &  Sxhh2 &  Sxhh3
\\ \hline
   \includegraphics[angle=0,width=3.2cm]{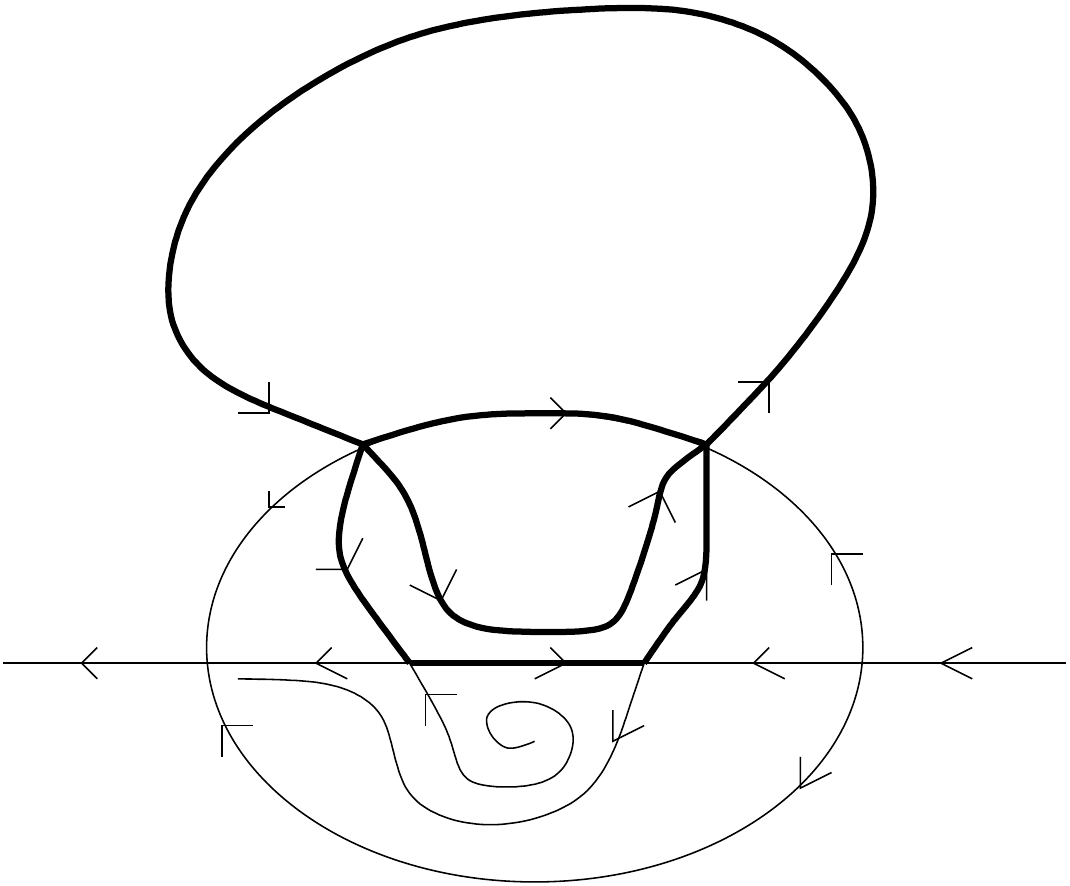}
&  \includegraphics[angle=0,width=3.2cm]{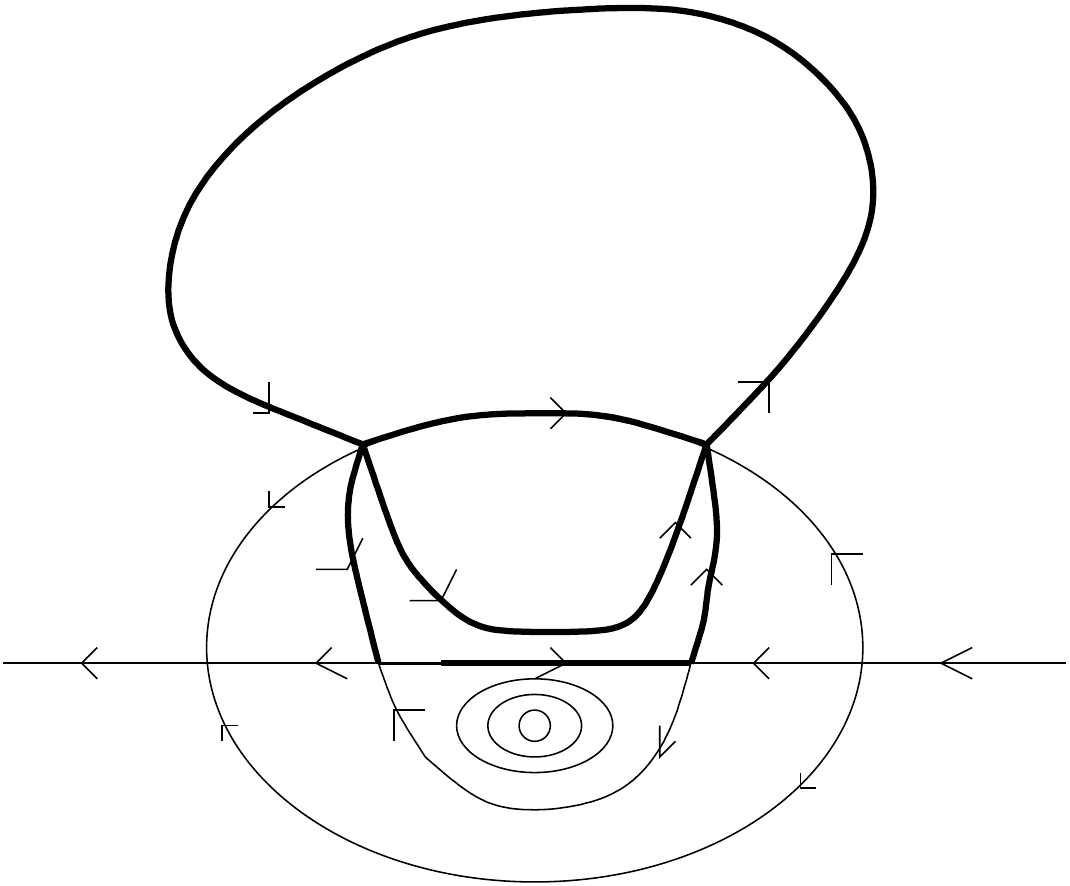}
&  \includegraphics[angle=0,width=3.2cm]{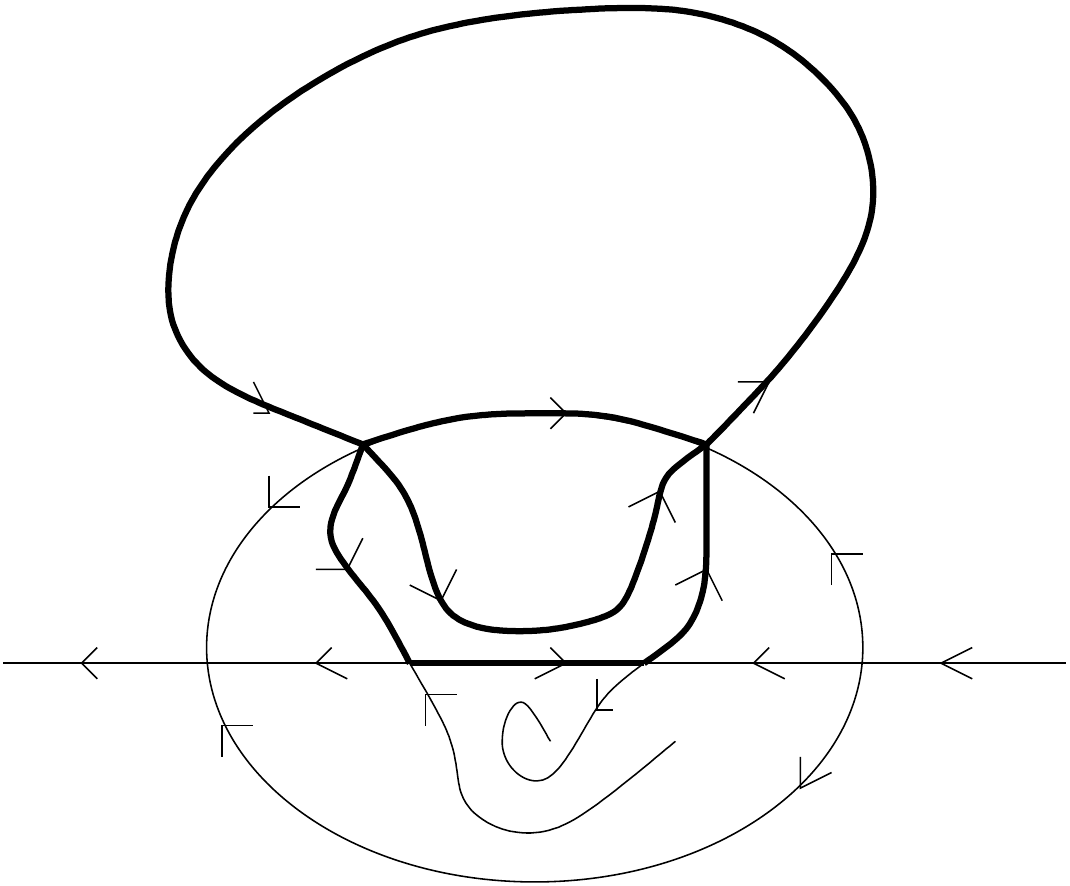}
\\ \hline  Sxhh4  &  Sxhh5  &  Sxhh6
\\ \hline
   \includegraphics[angle=0,width=3.2cm]{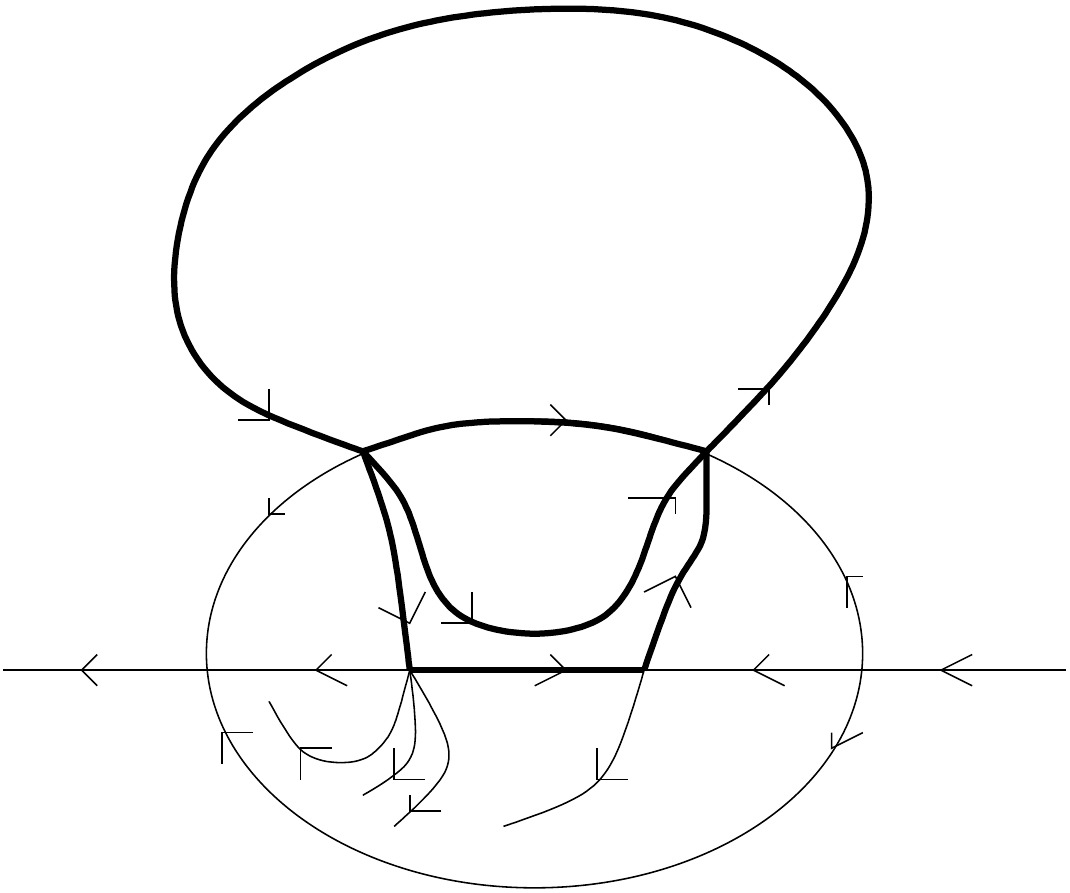}
&  &\includegraphics[angle=0,width=3.2cm]{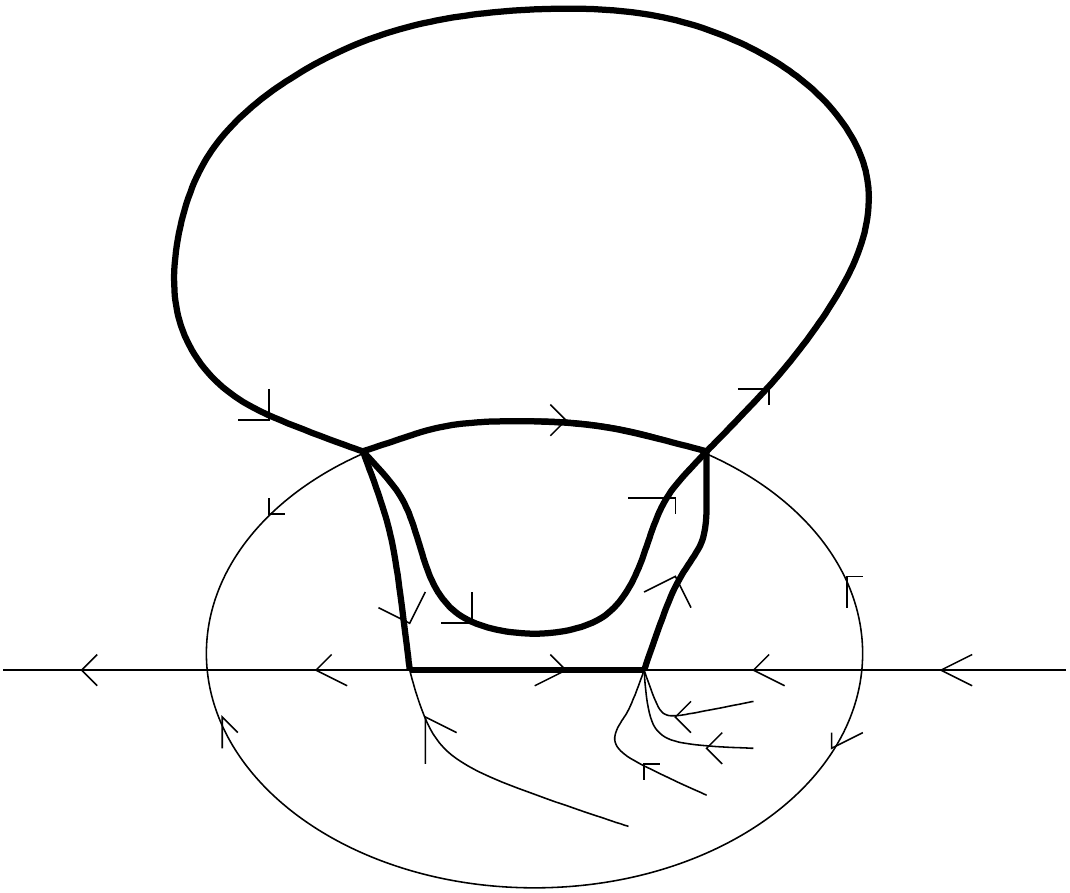}
\\  \hline  Sxhh7  &              &  Sxhh8
\\  \hline
   \includegraphics[angle=0,width=3.2cm]{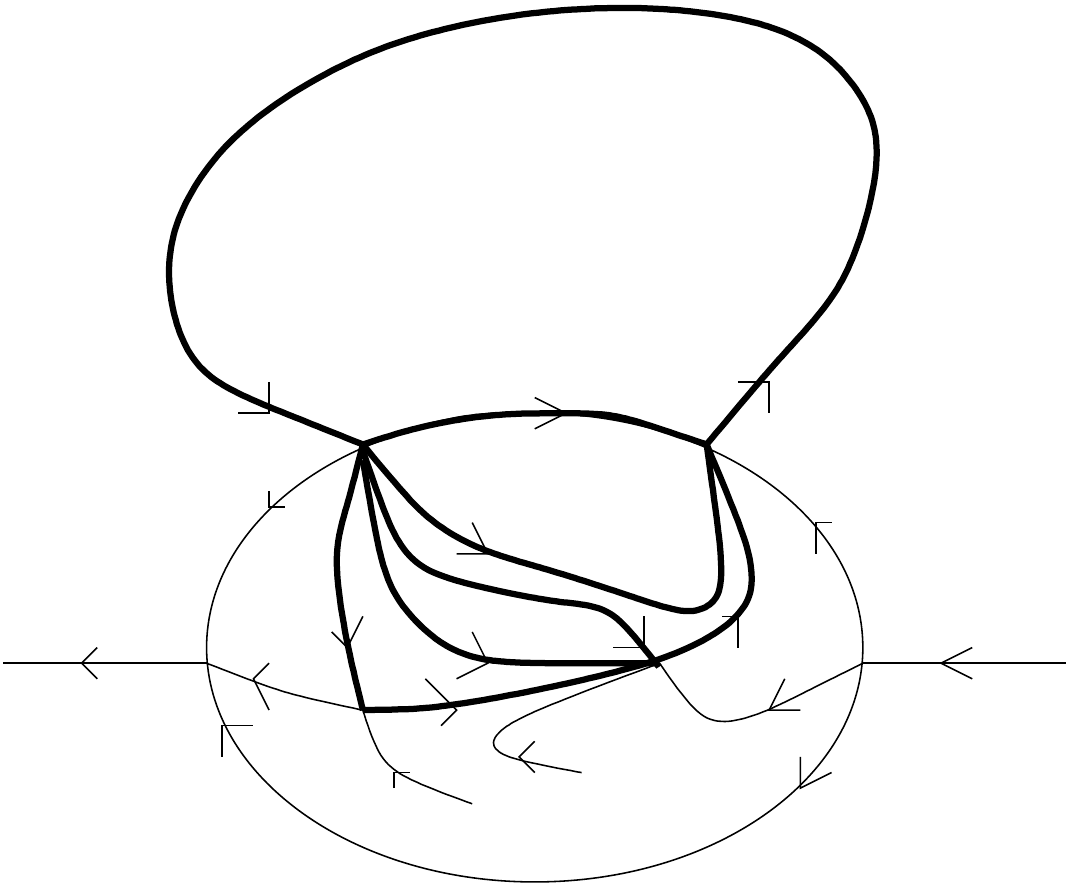}
&  &\includegraphics[angle=0,width=3.2cm]{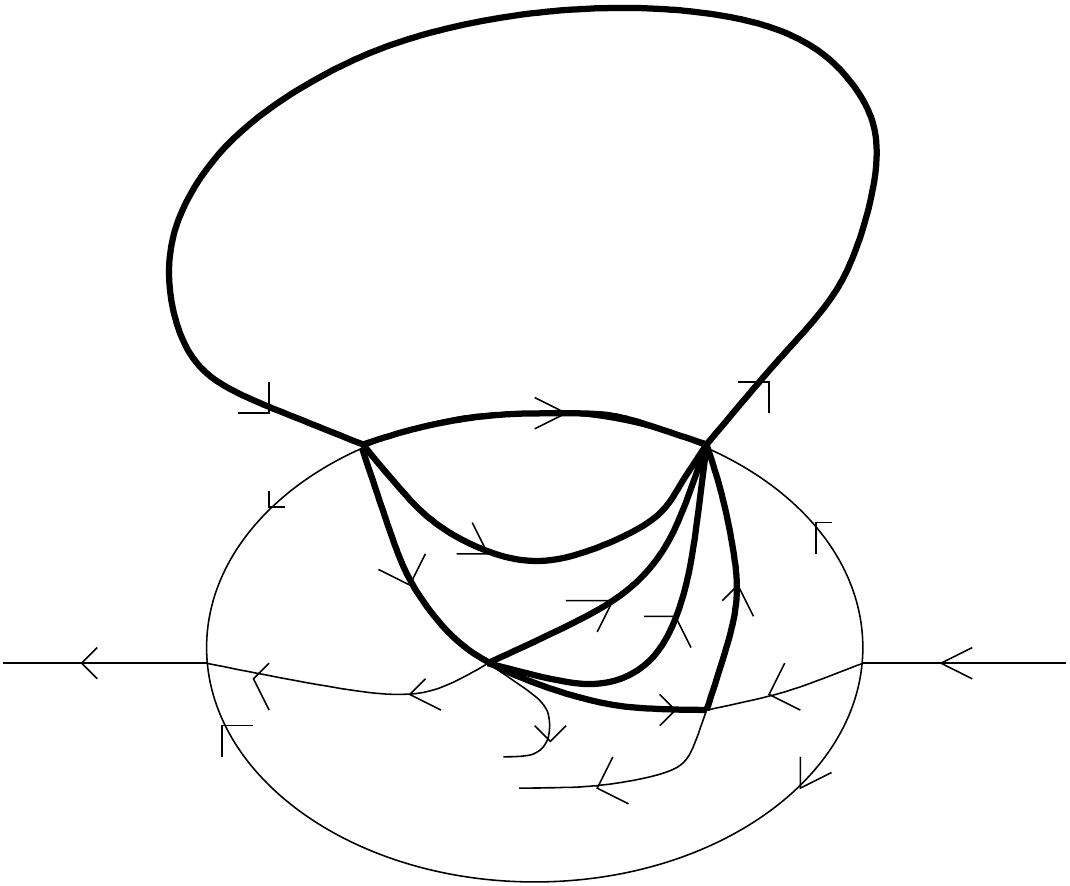}
\\   \hline  Sxhh9  &              &  Sxhh10
\\  \hline
\end{tabular}\end{center}
\caption{Convex limit periodic sets of hh-type for a graphic with
a nilpotent saddle}
\label{tab.shhconvex}\end{table}

\subsection{Proving the finite cyclicity of a limit periodic set }
The following argument will be used for proving the finite cyclicity of a limit periodic set: limit cycles correspond to fixed points of a  Poincar\'e return map defined on a section or, equivalently, to zeroes of a displacement map between two sections. The sections are 2-dimensional but, because of the invariant foliation, the problem can be reduced to a 1-dimensional problem and the conclusion follows by a derivation-division argument.
\bigskip

To compute the displacement map, we decompose the related transition maps between sections into compositions of Dulac maps in the neighborhood of the singular points and regular $C^k$ transitions elsewhere.

\subsection{Dulac maps} The Dulac maps have been computed in \cite{ZR}. There are two types of Dulac transitions. The first type of transition map goes from a section $\{r= r_0\}$ to a section $\{\rho = \rho_0\}$, or the other way around. This type of transition typically behaves as an affine map which is a very strong contraction or dilatation. The study of the number of  zeroes of a displacement map involving only Dulac maps of the first type is reduced to the study of the number of zeroes of a 1-dimensional map.

The second type of Dulac map is concerned with a transition from a section $\{\ov{y}=\ov{y}_0\}$ to, either a section $\{r=r_0\}$, or a section $\{\rho=\rho_0\}$. Here we only need the first type of Dulac map. We recall the precise results here.

\subsubsection{First type of Dulac map}\label{first_type_Dulac} We consider a Dulac map $D_i$ from a section $\Pi_i= \{\rho=\rho_0\}$ to a section $\Sigma_i=\{r=r_0\}$ in the neighborhood of a singular point $P_i$ (potentially following the flow backwards).
We decide to choose $(\nu,\tilde{y}_i)$ as coordinates on the sections $\Pi_i$ and $\Sigma_i$, where $\tilde{y}_i$ is a normalizing coordinate for the blow-up system in the neighborhood of $P_i$. The normal form near $P_i$ is given by
\begin{align}\begin{split}
\dot r&=r,\\
\dot \rho&=-\rho,\\
\dot{\tilde{y}}_i&= G(r, \rho, \tilde{y}_i),\end{split}\end{align}
where
\begin{equation}G(r, \rho, \tilde{y}_i)=\begin{cases} \tilde{y}_i (-\sigma +\varphi_i(\nu)), &\sigma_0\notin\Q,\\
 \tilde{y}_i (-\sigma +\varphi_i(\nu)+ f_i(r^p\tilde{y}_i)) + \eta_i(\nu)\rho^p,&\sigma_0=p\in \N,\\
  \tilde{y}_i (-\sigma +\varphi_i(\nu)+ f_i(r^p\tilde{y}_i^q)),&\sigma_0=\frac{p}{q},\: q>1\end{cases}\label{Normal_Form}\end{equation}
where
$$\sigma= \begin{cases}2(1-2a)= 2(1-2a_0)+\alpha, &i=3,4,\\
\frac{2a-1}{a}= \frac{2a_0-1}{a_0}+ \alpha,&i=1,2.\end{cases}$$

\begin{definition} The compensator $\omega$ is a univeral unfolding of the function $-\log x$, namely
 \begin{equation}\omega(x,\alpha)= \begin{cases} \frac{x^{-\alpha}-1}{\alpha},&\alpha\neq0,\\
-\log x, &\alpha=0.\end{cases}\label{def:omega}\end{equation}\end{definition}

The form of the Dulac map was first studied in \cite{ZR}. The following form is a refinement from  \cite{RR2}.

\begin{theorem}\label{thm_Dulac_type_1}
 We consider the Dulac map from the section $\{\rho= \rho_0\}$ to the section $\{r=r_0\}$, both parametrized by $(\tilde{y}_i,\nu).$
 Let $\nu_0=r_0\rho_0$ and \begin{equation}\bar\sigma_i=\sigma-\varphi_i(\nu)=\sigma_0+\alpha_i.\label{def_alpha}\end{equation} The $\tilde{y}_i$-component of the  transition map $D_i$ has the following expression:
\begin{enumerate}
\item If $\sigma_0\not\in \Q: $
\begin{equation}\label{eq20d}
D_i(\tilde{y}_i,\nu)=\Big(\frac{\nu}{\nu_0}\Big)^{\bar\sigma} \tilde{y}_i.
\end{equation}
\item    If $\sigma_0=\frac{p}{q}\in \Q$ with $(p,q)=1$:

\begin{equation}\label{eq21d}
D_i(\tilde{y}_i,\nu)=\eta_i(\nu)\rho_0^p\Big(\frac{\nu}{\nu_0}\Big)^{\bar\sigma}\omega\Big(\frac{\nu}{\nu_0},\alpha_i\Big)+\Big(\frac{\nu}{\nu_0}\Big)^{\bar\sigma}\Big(\tilde{y}_i+\phi_i(\tilde{y}_i,\nu,)\Big),
\end{equation}
where
\begin{itemize}
\item $\phi_i=O\left(\nu^{p+q\alpha_i}\omega^{q+1}\left(\frac{\nu}{\nu_0},\alpha_i\right)|\ln \nu|\right)$ and for any integer $l\geq 2,$   $\phi_{\mu,\sigma}$ is of class ${\cal C}^{l-2}$   in $\left(\tilde{y}_i,\nu^{1/l},\nu^{1/l}\omega\left(\frac{\nu}{\nu_0},\alpha_i\right), \nu,\mu,\sigma\right)$;
\item $\eta_i$  is as  in \eqref{Normal_Form}. In particular, $\eta_i\equiv 0$ when $\sigma_0\not\in \N.$
\end{itemize}
\end{enumerate}
\end{theorem}

\begin{remark}\label{remark} It follows from the form of $\phi$ as a function of class ${\cal C}^{l-2}$ on the generalized monomials $\tilde{y}_i$, $\nu^{1/l}$ and $\nu^{1/l}\omega\left(\frac{\nu}{\nu_0},\alpha_i\right)$ that all its derivatives with respect to $\tilde{y}_i$ of small order are $O(\nu^\beta)$ for some $\beta>0$. We say that $\phi$ has property $J$. \end{remark}

\subsection{Dulac map near a hyperbolic or semi-hyperbolic point}
When considering limit periodic sets, we will have additional singular points on them, and their associated Dulac maps. These can be explicitly calculated when the system is in $C^k$ normal form. We recall very briefly the form of these Dulac maps.

\begin{theorem}\label{Dulac_saddle} We consider a polynomial normal form for a family depending on a multi-parameter $A$, in the neighborhood of a hyperbolic saddle point with eigenvalues $\lambda_1(A)>0, -\lambda_2(A)<0$. The \emph{hyperbolicity ratio} is defined as the quotient $\tau=\frac{\lambda_2(A)}{\lambda_1(A)}$. If the system near the saddle has the following $C^k$ normal form for $A$ close to $A_0:$\begin{align}\begin{split}
\dot x&=\lambda_1(A)x,\\
\dot y&=-\lambda_2(A)y( 1+ Q(x,y)),\end{split}\end{align}
with $$Q(x,y)=\begin{cases} 0, &\tau(A_0)\notin \Q^+,\\
\sum_{i=1}^K c_i(A) (x^py^q)^i, &\tau(A_0)= \frac{p}{q},
\end{cases}$$
then the Dulac map from $\{y=Y_0\}$ to $\{x=X_0\}$ is of the form
$$ D_A(x)= Y_0X_0^{-\tau(A)} x^{\tau(A)} (1+\phi(x,A)),$$ where $\phi$ has the property $I$ of Mourtada given in Definition~\ref{def_Property_I} below.
Note that $\phi\equiv 0$, when $\tau(A_0)\notin \Q$.

In the particular case $\tau(A_0)=1$, we need the more refined form
$$ D_A(x)= Y_0X_0^{-\tau(A)} (x+ \alpha x\omega(x,\alpha) + \phi(x,A))$$
where $\omega$ is the compensator defined in \eqref{def:omega}, $\tau=1-\alpha$, and $\phi$ has the property $I$ of Mourtada, with $\phi(x,A)=O(x^{1+\delta})$ for some $\delta>0$.
\end{theorem}

\begin{definition}\label{def_Property_I} A function $\phi(y,A)$ has the property (I) of Mourtada if $\phi$ is $C^K$ for some $K$ on $(0,y_0)\times W$, where $W$ is a neighborhood of $A_0$ in $A$-space, and if there exists some neighborhood $W'$ of the origin in $A$-space such that for all $0\leq j\leq K$,
$$\lim_{y\to0} y^i\frac{\partial^j\phi}{\partial y^j}(y,\lambda)=0,$$
uniformly for $\lambda \in W'$. \end{definition}

\begin{theorem}\label{Dulac_saddle-node} \cite{DRR2}  We consider a polynomial normal form for a family depending on a multi-parameter $A$ in the neighborhood of a saddle-node with eigenvalues $0, -\lambda<0$, for $A=A_0$. If the system has the following normal form near the saddle-node \begin{align}\begin{split}
\dot x&=(x^2+\eta(A))( 1+ C(A)x^2)= F(x),\\
\dot y&=-\lambda y,\end{split}\end{align}
with $\eta(A_0)=0$, then \begin{enumerate}
 \item {\bf Case of central transition:} for $\eta>0$, the Dulac map from $\{x=-X_0\}$ to $\{x=X_0\}$ is linear of the form $D_A(y) = \eps(A)y$, with $\eps(A)>0$ exponentially small in $\sqrt{\eta}$;
\item  {\bf Case of stable-center transition:} the Dulac map $D_A(x)$ from $\{y=Y_0\}$ to $\{x=X_0\}$ is flat in $x$, as well as all its partial derivatives in $x$ and in the parameters.

\end{enumerate}
\end{theorem}

\section{Finite cyclicity of convex graphics through a nilpotent saddle of multiplicity $3$}

It was shown in \cite{ZR} that a graphic through a nilpotent saddle of codimension $3$ has finite cyclicity as soon as the first return map along the graphic has a derivative different from $1$. This excludes the value $a_0=-\frac12$ in \eqref{normal_form_family}.
This hypothesis was only used in studying the finite cyclicity of the limit periodic sets in $Sxhh1$ and $Sxhh5$. We now consider the case $a_0=-\frac12$.  We show that all limit periodic sets in $Sxhh1$ have finite cyclicity.
Under the additional hypothesis that the line on the blow-up sphere is a fixed connection, we also show that all limit periodic sets in $Sxhh1$ have finite cyclicity.

\begin{theorem}\label{thm:sxhh1} We consider a convex graphic through a nilpotent saddle of multiplicity 3 with $a_0=-\frac12$ and such that the derivative of the first return map $\gamma^{\ast}=P'(0)\neq1$. Then all limit periodic sets in Sxhh1 have finite cyclicity.\end{theorem}
\begin{proof} Without loss of generality we can suppose that the limit periodic set $\Gamma$ joins $P_3$ and $P_4$ (see Figure~\ref{fig.strat}). Note that the finite cyclicity of the upper boundary graphic of $Sxhh1$ was
proved in \cite{ZR}. Therefore, we only need to prove that the
intermediate graphics $Sxhh1b$ and the lower boundary graphic of $Sxhh1c$ have
finite cyclicity. The only place where the hypothesis  $a_0\neq-\frac12$ was used in \cite{ZR} is when the hyperbolicity ratio $\tau(\ov{M}_0)$ (i.e. the quotient of minus the negative eigenvalue to the positive one) is equal to $1$ at the saddle point of \eqref{family_rescaling}. Since the divergence of \eqref{family_rescaling} is identically equal to $\bar{\mu}_3$ for $a_0=-\frac12$, we need only consider the case $A_0=(-\frac12, \ov{\mu}_1,\ov{\mu}_2,0,0)$.
\begin{figure}[!h]
\begin{center}
\includegraphics[angle=0, width=0.37\textwidth]{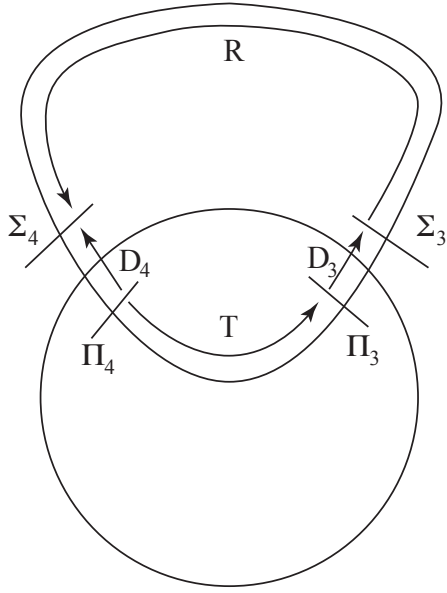}
\caption{Transition map for the hh-graphics of saddle type}
\label{transition_map}
\end{center}
\end{figure}

Let $\ov{\Gamma}$ be any intermediate or lower boundary graphic of
$Sxhh1$. To study its cyclicity, we take coordinates $(r,\rho,\ov{y}_i) $ in the neighborhood of $P_i$, $i=3,4$, where $r= x$ (resp. $-x$) for $P_3$ (resp. $P_4$) and $\ov{y}_i=\ov{y}-\frac{1-2a}{2}$ (hence $\ov{y}_i=0$ at $P_i$). A $C^k$-change of coordinates to normal form in the neighborhood of $P_i$ can be taken of the form $\tilde{y}_i=\ov{y}_i + f_i(r,\rho,\ov{y}_i)$. Let us   take sections $\Sigma_i=\{r=r_0\}$ and $\Pi_i=\{\rho=\rho_0\}$  as shown in Fig.~\ref{transition_map} in the
normal form coordinates $(r, \rho, \tilde{y}_i)$ in the
neighborhood of the singular point $P_i$ ($i=3, 4$). We will study the
displacement map $L:\Pi_4\longrightarrow\Sigma_3$ defined by
\begin{equation}
L=R^{-1}\circ D_4-D_3\circ T,\label{eq:L}
\end{equation}
where $R: \Sigma_3\longrightarrow\Sigma_4$ and $T: \Pi_4\longrightarrow \Pi_3$ are the transition maps along the
regular orbits in the normal form coordinates, and $D_i:\Pi_i\longrightarrow \Sigma_i$ are the Dulac maps.
We will study the maximum number of small
roots of $L=0$.

We decide to choose $(\nu,\tilde{y}_i)$ as coordinates on the sections $\Pi_i$ and $\Sigma_i$. The maps $R$ and $T$ are two-dimensional but, since they preserve the $\nu$-coordinate, we will cheat a little and identify them with their second component  which depends on $\nu$, and which we denote $R_\nu$ and $T_\nu$. We denote by $L_\nu$ the corresponding second component of $L$ in \eqref{eq:L}. For $\nu\in [0,\nu_0)$, $R_\nu$ and $T_\nu$ are regular $C^k$-diffeomorphisms. Let $S_\nu=R_\nu^{-1}$.
The Dulac maps $D_i$ near $P_4$ (following the flow backwards) and near $P_3$ are calculated in Theorem~\ref{thm_Dulac_type_1}, with $\sigma_0 = 4$.

Let
$$\alpha_{34}= \ov{\sigma}_3- \ov{\sigma}_4= \nu O(1).$$ The map $L_\nu$ has the form
\begin{equation} L_\nu(\tilde{y})= m_0(\nu, \lambda) + \left(\frac{\nu}{\nu_0}\right)^{\ov{\sigma}_3}\left[T_\nu'(0)-  S_\nu'(0)\left(\frac{\nu}{\nu_0}\right)^{-\alpha_{34}}+O(\nu)\right]\tilde{y}_4 +
\left(\frac{\nu}{\nu_0}\right)^{\ov{\sigma}_3} o(\tilde{y}_4). \label{equation_L}\end{equation}
It is clear that an intermediate graphic has cyclicity $1$ as soon as $T_\nu'(0)-  S_\nu'(0)\nu^{\ov{\sigma}_4- \ov{\sigma}_3}$ is bounded away from $0$ for $A$ in a neighborhood
of $A_0=(-\frac12, \ov{M}_0,0)$. This is precisely the case when $T_\nu'(0)$ is close to $1$.
Indeed, we know that $S_\nu'(0)\neq1$. Also,
$$\left(\frac{\nu}{\nu_0}\right)^{-\alpha_{34}}= e^{-\alpha_{34} \log(\nu/\nu_0)} = 1+O(\nu^{1-\delta})$$
for some small $\delta$, since $\alpha_{34} = O(\nu)$.
Hence, it suffices to show that $T_0'(0)=1$ when  $A_0=(-\frac12, \ov{\mu}_1,\ov{\mu}_2,0,0)$. We show the stronger property that $T_0\equiv id$ for such an $A_0$. For this purpose, we use that the system \eqref{family_rescaling} is Hamiltonian for $a=-\frac12$ and $\ov{\mu}_3=0$: the trajectories are level curves $H(\ov{x},\ov{y})=C$ of the Hamiltonian
$$H(\ov{x},\ov{y})= \frac 12 \ov{y}^2-\frac12 \ov{x}^2\ov{y} +\ov{\mu}_2\ov{y}-\ov{\mu}_1\ov{x}.$$ Hence, we must explain the link between the constant $C$ and the corresponding normalizing coordinates $\tilde{y}_3$ (resp.  $\tilde{y}_4$) on $\Pi_3$ (resp. $\Pi_4$). For this, we must not forget that the family rescaling has been obtained by putting $\rho=1$ after the blow-up.
For $r=0$, the system in $(\rho, \ov{y})$-coordinates is given by
\begin{align}\begin{split}
\dot \rho&= \mp \rho(\ov{y}-\frac12 +\ov{\mu}_2\rho^2),\\
\dot{\ov{y}}&=\pm2 \ov{y}\mp2 \ov{y}^2 \mp 2\ov{\mu}_2\ov{y}\rho^2+\ov{\mu}_1\rho^3,\end{split} \label{P_i_r=0}\end{align}
where the sign $+$ (resp. $-$) comes from putting $\ov{x}=+1$ (resp $\ov{x}=-1$).
The function $\rho^{-5}$ is an integrating factor of \eqref{P_i_r=0}, which yields first integrals
$$\ov{H}_\pm = \frac{\ov{y}^2}{2\rho^4} -\frac{\ov{y}}{2\rho^4}+\ov{\mu}_2\frac{\ov{y}}{\rho^2}\mp\ov{\mu}_1\frac1{\rho}.$$
We need to localize at $P_3$ and $P_4$ by letting $z= \ov{y} -1$. Then
$$\ov{H}_\pm = \frac{z^2}{2\rho^4} +\frac{z}{2\rho^4}+\ov{\mu}_2\frac{z+1}{\rho^2}\mp\ov{\mu}_1\frac1{\rho},$$
which means that the trajectories are given by
$$Z=\frac{z^2}2+ \frac{z}2 +\ov{\mu}_2(z+1)\rho^2\mp\ov{\mu}_1\rho^3= C_\pm \rho^4,$$
The change of coordinate $z\mapsto Z$ is invertible for small $z$ and is precisely the normalizing coordinate.
Then it is easy to see that on sections $\Pi_3$ and $\Pi_4$ with common equation $\{\rho=\rho_0\}$ we have $\tilde{y}_3=C_+\rho_0^4$ and $\tilde{y}_4=C_-\rho_0^4$, and also that $C_+= C=C-$ for a given trajectory. Hence $T_0\equiv id$, which means that $T'$ is close to $1$ for $A$ close to $A_0$ in the neighborhood of the limit periodic set.

\

We now only need to consider the lower graphic  Sxhh1c for $A_0=(-\frac12, \ov{\mu}_1,\ov{\mu}_2,0,0)$. Let $\tau(M)=1-\alpha $ be the hyperbolicity ratio at the saddle point of \eqref{family_rescaling}.

Using Theorem~\ref{Dulac_saddle}, the regular transition near the hyberbolic saddle in suitable normal form coordinates has the form
$$V_\nu(\tilde{y})= m_0(A)+m_1(A)\alpha\omega(\tilde{y},\alpha)\tilde{y} + m_2(A)\tilde{y}+ O\left(\tilde{y}^2\,\omega(\tilde{y},\alpha)\right),$$
with $m_0(A_0)=m_1(A_0)= m_2(A_0)-1=0$,
which yields that the transition map $T_\nu$ has the form
\begin{align}\begin{split}T_\nu(\tilde{y}_3)&= n_0(A)+n_1(A)\alpha\tilde{y}_3\omega(\tilde{y}_3,\alpha)(1+ \phi_1(\tilde{y}_3, \alpha))\\
&\qquad\quad + n_2(A)\tilde{y}_3(1+ \phi_2(\tilde{y}_3, \alpha))+ O\left(\tilde{y}^2\,\omega(\tilde{y},\alpha)\right),\end{split}\label{lower_lps}\end{align}
with $n_0(A_0)=n_1(A_0)= n_2(A_0)-1=0$, where the functions $\phi_j$ have the property (I) of Mourtada (see Definition~\ref{def_Property_I}).

This yields that $L_\nu(\tilde{y}_3)$ has the  form \begin{align}\begin{split}L_\nu(\tilde{y}_3)&= \tilde{n}_0(A,\nu)+n_1(A,\nu)\left(\frac{\nu}{\nu_0}\right)^{\ov{\sigma}_3}
\alpha\tilde{y}_3\omega(\tilde{y}_3,\alpha)(1+ \psi_1(\tilde{y}_3,\nu)) \\
&\qquad+\left(\frac{\nu}{\nu_0}\right)^{\ov{\sigma}_3}\left[n_2(A_\nu) - S_\nu'(0)\left(\frac{\nu}{\nu_0}\right)^{\alpha_{34}} + O(\nu)\right] \tilde{y}_3(1+ \psi_2(\tilde{y}_3,\nu)),\end{split}\label{L_tilde}\end{align}
where $\tilde{n}_0(A_0,0)=\alpha(A_0,0)=0$. Let $\tilde{n}_2(A, \nu)=n_2(A_\nu) - S_\nu'(0)\left(\frac{\nu}{\nu_0}\right)^{\alpha_{34}} + O(\nu)$, then we have $\tilde{n}_2(A_0,0)\neq0$. $\psi_1,\psi_2$ are finite sums of products of functions with property (I) or (J).

By Rolle's theorem, the number of zeroes of $L_\nu$ is at most $1$ plus the number of zeroes of $N_{1,\nu}(\tilde{y}_3)=\left(\frac{\nu}{\nu_0}\right)^{-\ov{\sigma}_3}\frac{dL}{d\tilde{y}_3}(\tilde{y}_3)$. Considering that the derivative of $\omega(\tilde{y}_3,\alpha)$ is $1+\alpha \omega(\tilde{y}_3,\alpha)$, we have
$$N_{1,\nu}(\tilde{y}_3)=n_1(A,\nu)[(1-\alpha) \omega(\tilde{y}_3,\alpha)-1](1+ \xi_1(\tilde{y}_3,\nu)) + \tilde{n}_2(A,\nu)(1+ \xi_2(\tilde{y}_3,\nu)),$$
where $\xi_1,\xi_2$ are finite sums of functions with property (I) and (J).
The number of zeroes of $N_{1,\nu}(\tilde{y}_3)$ is the same as the number of zeroes of $$N_{2,\nu}(\tilde{y}_3)= \frac{N_{1,\nu}(\tilde{y}_3)}{[(1-\alpha) \omega(\tilde{y}_3,\alpha)-1](1+ \xi_1(\tilde{y}_3,\nu))}.$$
By Rolle's theorem again, this number is  at most $1$ plus the number of zeroes of $N_{3,\nu}(\tilde{y}_3)=\frac{dN_{2,\nu}}{d\tilde{y}_3}(\tilde{y}_3)$, given by
$$N_{3,\nu}(\tilde{y}_3)= -\tilde{n}_2(A,\nu)\frac{(1-\alpha)\tilde{y}_3^{-1-\alpha}}{[(1-\alpha) \omega(\tilde{y}_3,\alpha)-1]^2}(1+\chi_2(\tilde{y}_3,\nu))\neq0,$$
with $\chi_2$ a sum of functions with property (I) and (J), since it is standard that  $x^n\omega(x,\alpha)$ is small for positive $n$ and small $(x,\alpha)$.
\end{proof}

\begin{theorem}\label{thm:sxhh5} We consider a convex graphic through a nilpotent saddle of multiplicity 3 with $a_0=-\frac12$ passing through the points $P_3$ and $P_4$ of the blow-up, and such that the derivative of the first return map $\gamma^{\ast}=P'(0)\neq1$. We also suppose that there is a fixed connection on the blow-up sphere along a line joining $P_1$ and $P_2$ (corresponding to $\mu_1=0$ in \eqref{normal_form_family}. Then all limit periodic sets in Sxhh5 have finite cyclicity.\end{theorem}
\begin{proof} The proof is very similar to that of Theorem~\ref{thm:sxhh1}. When $\ov{\mu}_3\neq0$, then the product of the hyperbolicity ratios $\tau_1\tau_2$  at the two saddle points is different from $1$, and the finite cyclicity was proven in \cite{ZR}.
When $\ov{\mu}_3=0$, then the family rescaling \eqref{family_rescaling} is integrable, both because it is symmetric and Hamiltonian. Hence, for the intermediate limit periodic sets, the transition map $T_\nu$ is close to the identity.
As for the lower periodic set through the two saddle points, the transition map $T_\nu$ has the same form as in \eqref{lower_lps} with $\tau=\tau_1\tau_2= 1-\alpha$.
  \end{proof}

\begin{remark} We conjecture that the hypothesis that $\mu_1=0$ in Theorem~\ref{thm:sxhh5} can be dropped, but we have not been able to prove it. \end{remark}

\begin{corollary}\label{thm:saddle} We consider a convex graphic through a nilpotent saddle of multiplicity 3 with $a_0=-\frac12$ passing through the points $P_3$ and $P_4$ of the blow-up, and such that the derivative of the first return map $\gamma^{\ast}=P'(0)\neq1$. We also suppose that there is a fixed connection on the blow-up sphere along a line joining $P_1$ and $P_2$. Then the graphic has finite cyclicity.\end{corollary}
\begin{proof} All limit periodic sets except Sxhh1 and SXhh5 were proved in \cite{ZR}
to have finite cyclicity for any $a_0$ negative. And we have proved the finite cyclicity of SXhh1 and Sxhh5 in Theorems~\ref{thm:sxhh1} and \ref{thm:sxhh5}. \end{proof}

\section{Applications to quadratic systems}

\subsection{ Quadratic systems with a nilpotent singular point at infinity}

\begin{proposition}\label{thm.infty}
A quadratic system with a triple singular point of
saddle or elliptic type at infinity and a finite singular point
of focus or center type can be brought to the form
\begin{equation}
\left\{\begin{array}{ll}
\dot x&=\delta x-y+Bx^2\\
\dot y&=x+\gamma y+xy.
\end{array}\right.
\label{inf}
\end{equation}
The value of ``$a$"  in the corresponding normal form \eqref{normal_form_family} is $a=1-B$. Moreover \begin{enumerate}
\item When $B>1$, the singular point is a nilpotent saddle.
\item For $B\neq0,\frac12$, the system has an invariant parabola
\begin{equation}
y=(B-\frac{1}{2})x^2+(2-\frac{1}{B}) \delta x -
          \frac{\ B+(1-2B)\delta^2\ }{2B^2}
\label{inf.invariant}
\end{equation}
if  \begin{equation}\label{cond_parabola} \gamma B-(1-2B)\delta=0.\end{equation}
\item The nilpotent saddle point is of codimension $4$ when $B=\frac32$ (corresponding to $a=-\frac12$).
\item The integrability condition is
$\gamma=\delta=0$.
\end{enumerate}
\end{proposition}

\begin{proof}
We can suppose that the nilpotent singular point at infinity is
located on the y-axis,  the other singular point at infinity on the x-axis and the focus or center at the origin.
Then the system can be brought to the form
\begin{equation}
\left\{\begin{array}{ll}
\dot x&=\delta_{10} x +\delta_{01} y   +\delta_{20} x^2   +\delta_{11}xy\\
\dot y&=\gamma_{10} x  +\gamma_{01}  y   +\gamma_{11}  xy    +\gamma_{02} y^2.
\end{array}\right.
\label{inf.1}
\end{equation}
For the finite singular point to be a focus or center, we should have
$\delta_{10}\gamma_{01}-\delta_{01}\gamma_{10}>0$.

Localizing the system \eqref{inf.1} at the singular point at infinity
on y-axis by
$v=\frac{x}{y}, \ \ z=\frac{1}{y}$, we have
\begin{equation}
\left\{\begin{array}{ll}
\dot v&=(\delta_{11}-\gamma_{02})v+\delta_{01} z+(\delta_{20}-\gamma_{11})v^2
        +(\delta_{10}-\gamma_{01})vz-\gamma_{10}v^2z\\
\dot z&=z(-\gamma_{02}-\gamma_{01} z -\gamma_{11}v-\gamma_{10} vz)
\end{array}\right.
\label{inf.2}
\end{equation}
The singular point $(0,0)$ of system \eqref{inf.2} is nilpotent,
if $\delta_{11}=\gamma_{02}=0$. It is triple if
$\gamma_{11}(\delta_{20}-\gamma_{11})\neq 0$. By a rescaling and still using the original
coordinates $(x,y)$, we obtain the system \eqref{inf}.

By a
transformation tangent to $(v,z)\mapsto(-V,z)$ and a time rescaling, we can bring system \eqref{inf.2} into the
$C^\infty$-equivalent form
\begin{equation}
\begin{cases}
\dot V=Z\\
\dot Z=(B-1)V^3- \gamma(B-1)^2 V^4+ O(V^5)\\
\qquad + Z\Big[
        (3-2B)V-\gamma(B-1)(B^2-2B+4) V^2+O(V^3)\Big]+Z^2O(|(V,Z)|^3).
\end{cases}
\label{inf.4}
\end{equation}
Then $\eta= -\gamma(B-1)^2(5B^2-4B+11)$ in
                  \eqref{normal_form_DRS} does not vanish when $\gamma\neq0$ and  $B>1$.
Also $b=3-2B$ vanishes for $B=\frac32$.              \end{proof}

\subsection{Finite cyclicity of graphics with a nilpotent point of saddle-type inside quadratic systems}

\begin{figure}[!h]
\begin{center}
\includegraphics[width=4.5cm]{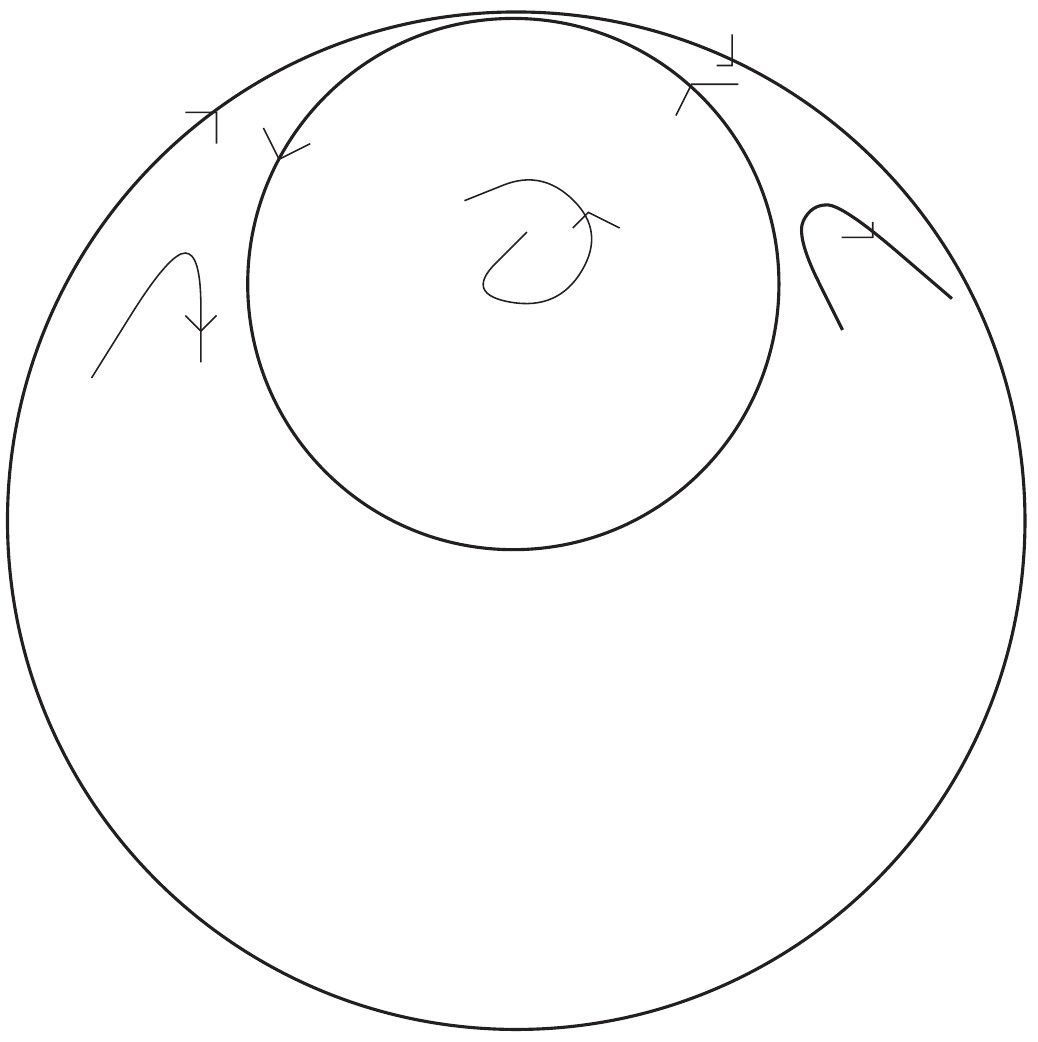}
\caption{The graphic $(I_{12}^{1})$.}\label{I_12_1}
\end{center}
\end{figure}

\begin{theorem}\label{thm:I_12^1} The graphic $(I_{12}^{1})$ (Figure~\ref{I_12_1}) has finite cyclicity inside quadratic systems. \end{theorem}

\begin{proof} The graphic $(I_{12}^{1})$ is an hh-type graphic with a nilpotent saddle of multiplicity 3
at infinity and an invariant parabola as shown in Fig \ref{I_12_1}.

By Theorem~\ref{thm:saddle}, to prove the finite cyclicity of $(I_{12}^{1})$, we only need to
check that the first return map $P$ of the system \eqref{inf} along the invariant parabola \eqref{inf.invariant} under condition \eqref{cond_parabola}
satisfies $\gamma^{\ast}=P'(0)\neq1$ when $\gamma\neq0$. Along the invariant parabola \eqref{inf.invariant}, we have
\begin{flalign*}
\begin{array}{ll}
P'(0)
&=\displaystyle{\exp\left(\int^\infty_{-\infty}\ div\; dt\right)}\\
&=\displaystyle{\lim_{x_0\to \infty}
     \exp\left(\int^{x_0}_{-x_0}
     \frac{(1+2B)x+\frac{(1-B)\delta}{B}}{\quad \frac12x^2+\frac{(1-B)\delta}{B} x+ \frac{(1-2B)\delta^2+B}{2B^2}\quad} dx\right)}\\
&=\displaystyle{\lim_{x_0\to \infty}\left[
  \left(\frac{\ -B^2x_0^2+2\delta B(B-1)x_0+\delta^2(2B-1)-B\ }
             {\ -B^2x_0^2-2\delta B(B-1)x_0+\delta^2(2B-1)-B\ }
  \right)^{1+2B}\right.}\\
  &\qquad\displaystyle{\left.
  \exp\left(4\delta B^{1/2}(1-B)(1-B\delta^2)^{-1/2}
         \arctan\frac{\ -Bx+(B-1)\delta\ }{\sqrt{B(1-B\delta^2)}} \right)\Big|^{x_0}_{-x_0} \right] }\\
&=\displaystyle{ \exp\left(4\pi \delta B^{1/2}(1-B)(1-B\delta^2)^{-1/2}\right)  }\neq 1,
\end{array}
\end{flalign*}
when $\delta\neq0$ and $B\neq1$. (Note that $1-B\delta^2>0$ is the condition that the system has no singular point on the invariant parabola.)
\end{proof}

\begin{figure}[!h]
\begin{center}
\includegraphics[width=4.5cm]{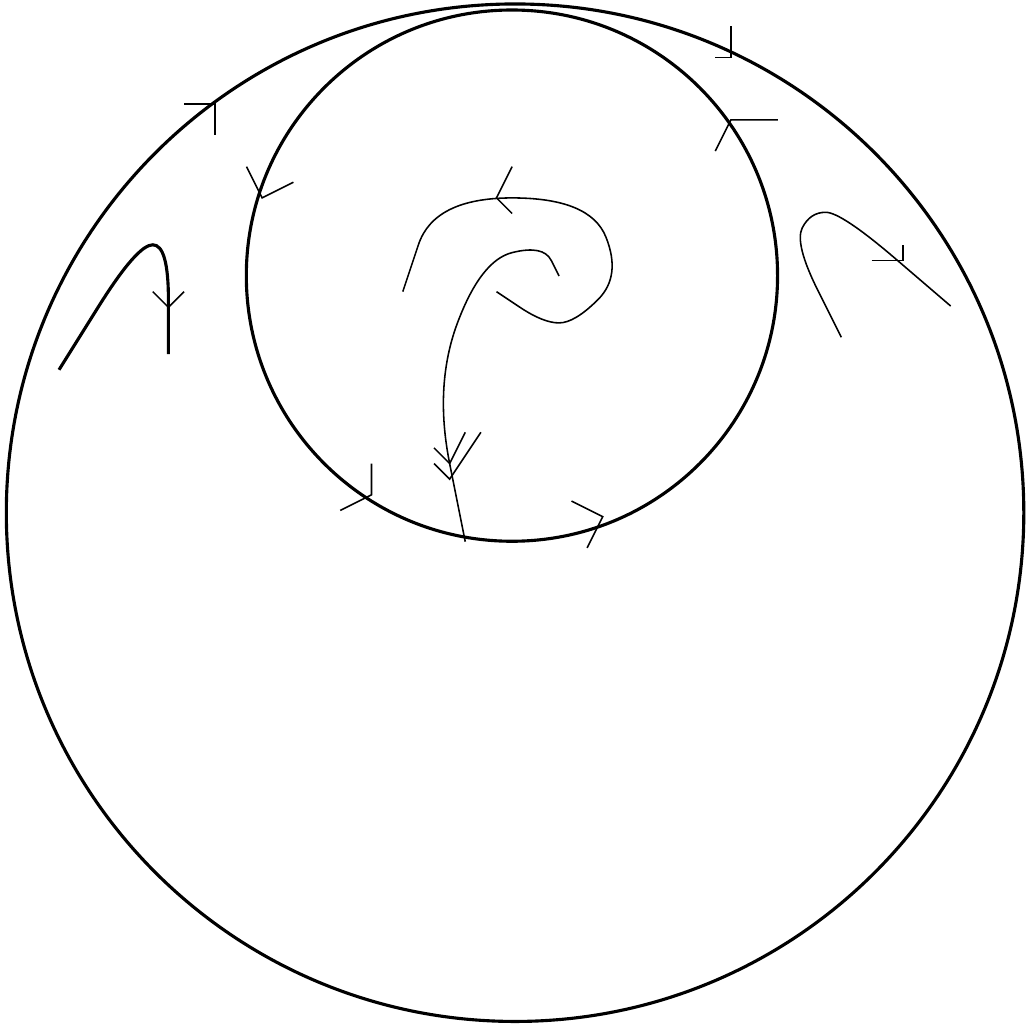}
\caption{The graphic $(I_{13}^{1})$.}\label{I_13_1}
\end{center}
\end{figure}

\begin{theorem}\label{thm:saddle_sn} The graphic $(I_{13}^1)$ (see Figure~\ref{I_13_1}) has finite cyclicity inside quadratic systems.
\end{theorem}
\begin{proof} This graphic is a convex graphic through a nilpotent saddle of multiplicity 3, and with a central transition through a saddle-node. In quadratic systems, such a graphic occurs when the nilpotent point is at infinity. Then $\mu_1=0$ in the unfolding, because the equator is invariant. This limits the number and complexity of the limit periodic sets to be considered.
Without loss of generality, we can suppose that the saddle-node is attracting. The proof is an easy adjustement of that of Corollary~\ref{thm:saddle}. Indeed, by Theorem~\ref{Dulac_saddle-node},  the central transition through a saddle-node in normal form coordinates is linear with exponentially small coefficient $\eps(A)$ in the parameter unfolding the saddle-node.

Because of the restriction to quadratic systems (hence $\mu_1=0$) we need only consider the limit periodic sets occurring in Sxhh1-Sxhh8 of Table~\ref{tab.shhconvex}, and the connection along the invariant line is always fixed. The  upper and intermediate graphics all have cyclicity one: indeed,  the first return map has a derivative much smaller than one because of the passage near the saddle-node by Theorem~\ref{Dulac_saddle-node}.

Hence, we need only consider the lower limit periodic sets. The cyclicity is one for Sxhh2c. Indeed, the global Poincar\'e return map has a derivative less than $1$, since the Dulac map near the attracting saddle-node on the blow-up sphere is flat (Theorem~\ref{Dulac_saddle-node}, case 2), and hence has a very small derivative. The same is true for Sxhh8c because the transition is fixed between the saddle and the saddle-node on the blow-up sphere. Indeed, since the stable-center transition near the saddle-node is flat, then the composition of three maps on the blow-up sphere (the passage near the saddle (given in Theorem~\ref{Dulac_saddle}) with  the regular transition between the saddle and the saddle-node and the stable-center transition near the saddle-node is flat.

We group the rest of the limit periodic sets  into classes and give sketchy arguments, since these are quite classical.

\medskip \noindent{\bf Sxhh1, Sxhh4, Sxhh5 and Sxhh6.} The argument is similar to the finite cyclicity of a graphic with a saddle-node with center transition and a hyperbolic saddle. The cyclicity is $1$ if the hyperbolicity ratio $\tau$ at the saddle for Sxhh1 (resp. the product $\tau$ of the hyperbolicity ratios at the two saddle points for Sxhh4 and Sxhh6)  is greater than one since the Poincar\'e return map has a derivative less than $1$.

When $\tau\leq1$, we consider the displacement map $L_\nu: \Sigma_4 \longrightarrow \Sigma$ (see Figure~\ref{transition_hh_SN}(a)), defined by $L_\nu=R_{3,\nu}\circ D_{3,\nu}\circ T_\nu\circ D_{4,\nu}^{-1}- D_\nu^{-1}\circ R_{4,\nu}^{-1}$. It has been shown in \cite{GR} that it is possible to choose normalizing coordinates on $\Pi$, such that $R_{4,\nu}$ is an affine map. Hence,  $D_\nu^{-1}\circ R_{4,\nu}^{-1}$ is an affine map, whose second derivative is identically zero. If $\tau<1$, then we directly see that
$L_\nu''(\tilde{y}_4)\neq0$, since $T_\nu(\tilde{y}_4) = \eps_0+C\tilde{y}_4^\tau+ O(\tilde{y}_4)$, with $C\neq0$.
If $\tau=1$, which occurs for $\mu_3=0$, then we can use exactly the same sections and arguments as in Theorem~\ref{thm:sxhh1} since the
family rescaling is integrable in this case.

\medskip \noindent{\bf Sxhh3.} The argument is similar to the finite cyclicity of a graphic with two saddle-nodes, one with center transition (the one on the blow-up sphere) and one with center-unstable transition. It involves using the Khovanskii method.
\begin{figure}
\begin{center}
\subfigure[Intermediate graphic]
{\includegraphics[height=6cm]{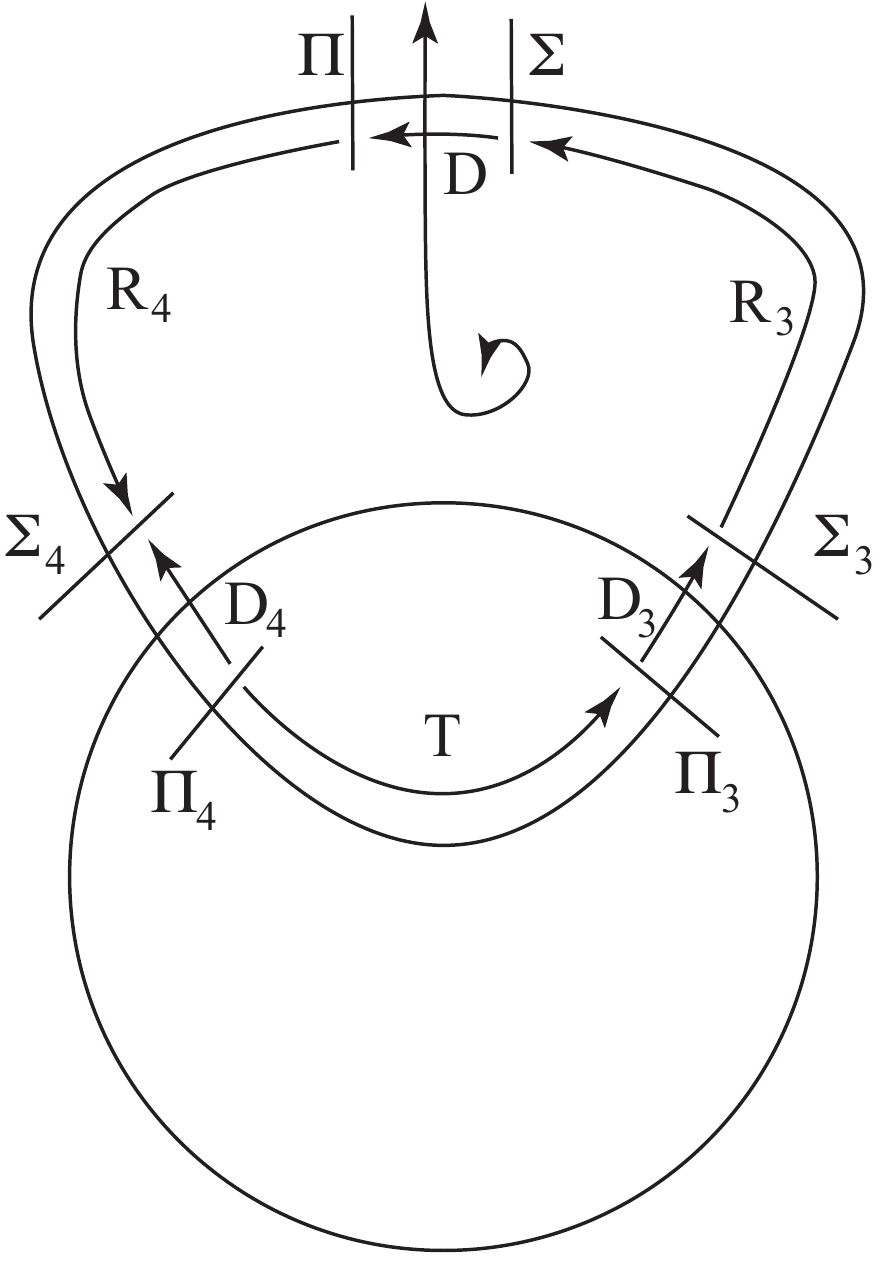}}\qquad
\subfigure[Sxhh3]
{\includegraphics[height=6cm]{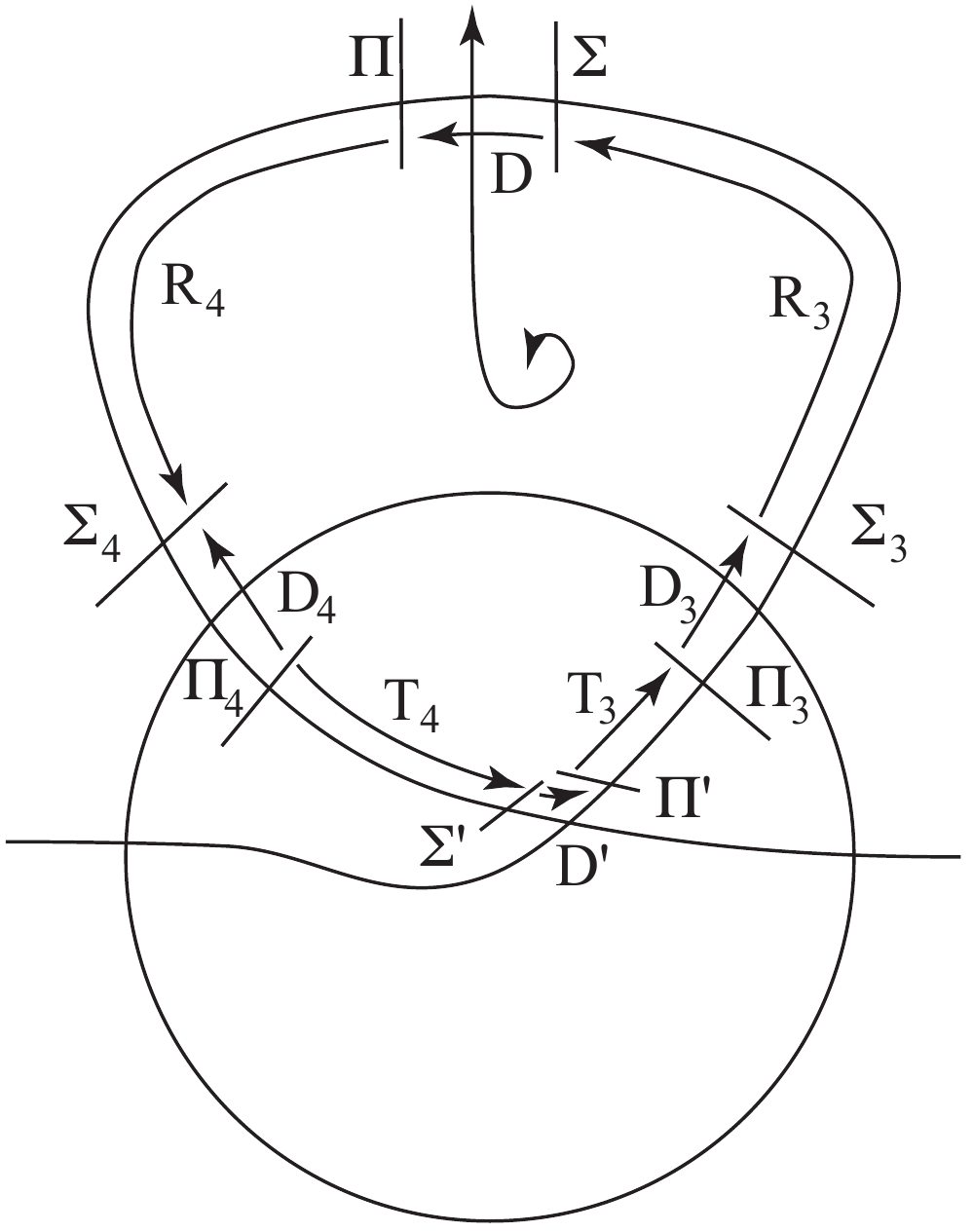}}\qquad
\subfigure[Sxhh7]
{\includegraphics[height=6cm]{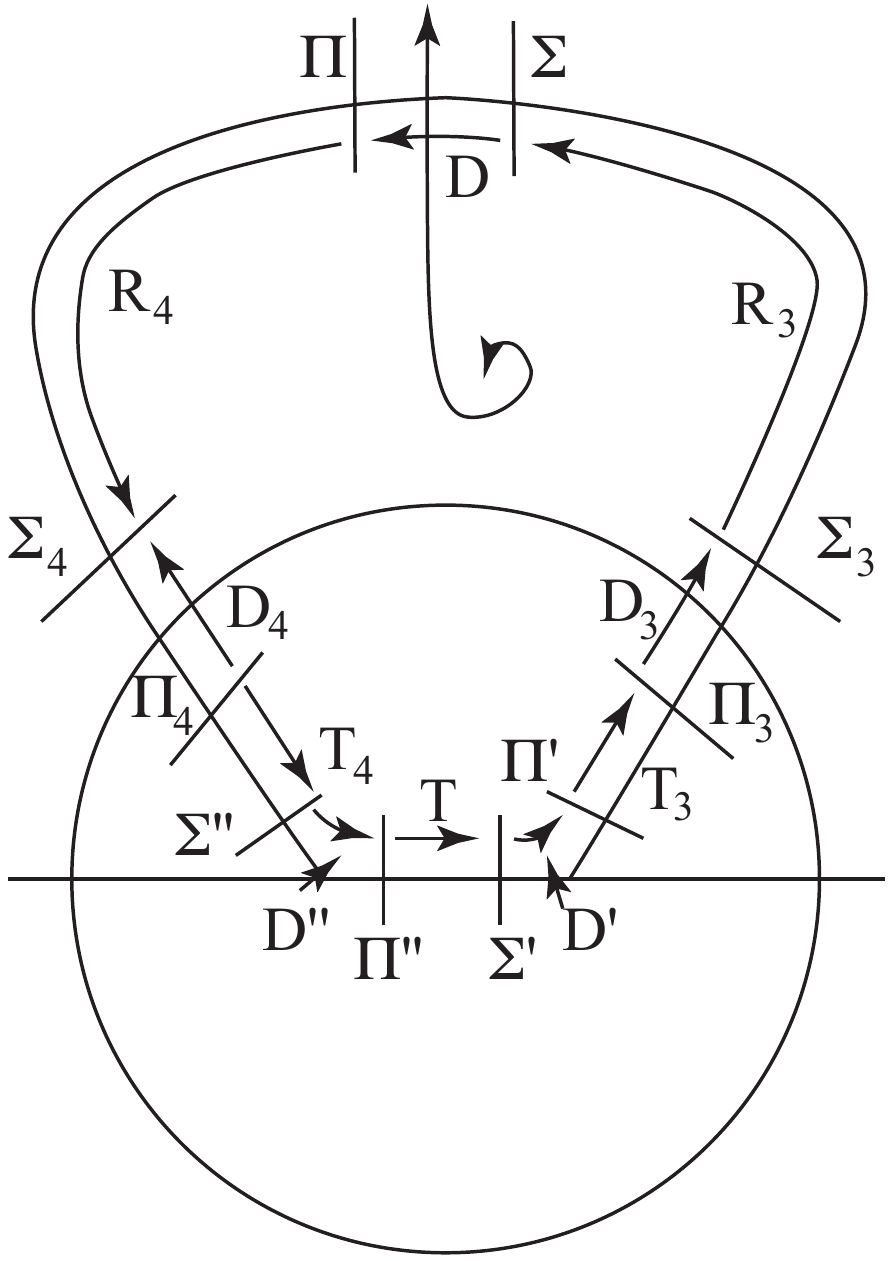}}
\caption{The sections for $(I_{13}^1)$.} \label{transition_hh_SN}\end{center}\end{figure}
Indeed, let $\Sigma'$ and $\Pi'$ be two sections in normal form coordinates at the entrance and exit of the saddle-node on the blow-up sphere (see Figure~\ref{transition_hh_SN}(b)), where $\Sigma'$ is parameterized by $z$ and $\Pi'$ by $w$. We replace  considering  the displacement map from $\Pi'$ to $\Sigma'$ by considering the equivalent system of two equations
\begin{equation}\begin{cases}
z = S_\nu(w), \\
z = D_\nu'^{-1}(w),
\end{cases}\label{Kho1}\end{equation}
where $S_\nu$ follows the flow forwards:
\begin{equation}S_\nu=T_{4,\nu}\circ D_{4,\nu}^{-1}\circ R_{4,\nu}\circ  D_\nu\circ R_{3,\nu}\circ D_{3,\nu}\circ T_{3,\nu}.\label{formula_S}\end{equation} The Taylor expansion of $S_\nu$ has the form $S_\nu(w)= \eps_0(A)+ \eps_1(A)w(1+ h(w,A))$, where $h(w,A)=O(w)$ is bounded and has property (J). Also $\eps_1(A)>0$, when the saddle-node has disappeared, a necessary condition for the existence of limit cycles.
Now, $D_\nu^{-1}$ is the Dulac map following the flow backwards near the saddle-node. The function $z=D_\nu^{-1}(w)$ is solution of the Pfaff equation $F_A(w)dz-zdw=0$, where $F_A(w)= (w^2+\eta(A))(1+C(A)w)$ and $F_A(w) \frac{\partial}{\partial w} + z\frac{\partial}{\partial z}$ is the normal form of the vector field in the neighborhood of the saddle-node. Hence, we replace the system \eqref{Kho1} by the system
\begin{equation}\begin{cases}
z = S_\nu(w), \\
\Omega=F_A(w)dz-zdw=0,
\end{cases}\label{Kho2}\end{equation}
Between two solutions of the system~\eqref{Kho2}, there exists on $z=S_\nu(w)$ a contact point of $\Omega$ with $z=S_\nu(w)$. Hence, the number of solutions is at most one plus the number of solutions of
\begin{equation}\begin{cases}
z = S_\nu(w), \\
z-S_\nu'(w)F_A(w) =0,
\end{cases}\label{Kho3}\end{equation}
which yields the 1-dimensional equation $V_A(w)=S_\nu(w)-S_\nu'(w)F_A(w)=0$. This equation has at most one small solution. Indeed,
$$V_A'(w)= \eps_1(A)\left[1+O(w)+O(\eta)\right] \neq0,$$
for small $w$ and $A$ sufficiently close to $A_0$.

\medskip \noindent{\bf Sxhh7.} We only need to adapt the argument done for Sxhh3.
We consider the sections in Figure~\ref{transition_hh_SN}(c). Since the connection between the saddle and the saddle-node is fixed on the blow-up sphere, this suggests taking for the displacement map, the map from  $\Pi'$ to $\Sigma''$, parametrized respectively by $z$ and $w$. As before, we consider the equivalent system of two equations
\begin{equation}\begin{cases}
z = S_\nu(w), \\
z = U_\nu(w),
\end{cases}\label{Kho4}\end{equation}
where $S_\nu$ is given by \eqref{formula_S} and $U_\nu= D_\nu''^{-1}\circ T_\nu^{-1}\circ D_\nu'^{-1}$.
Let $\tau(A)$ be the hyperbolicity ratio at the saddle point.
We have  \begin{equation}v=T_\nu\circ D_\nu''(z) = c(A) z^{\tau(A)}( 1+ \phi_1(z,A)),\label{equation_v_z}\end{equation} where $c(A)>0$ and $\phi_1$ has property (I).
As before, the function $v=D_\nu'^{-1}(w)$ is solution of the Pfaff equation $F_A(w)dv-vdw=0$, where $F_A(w)= (w^2+\eta)(1+C(A)w)$ and $F_A(w) \frac{\partial}{\partial w} + v\frac{\partial}{\partial v}$ is the normal form of the vector field in the neighborhood of the saddle-node.
Then, replacing \eqref{equation_v_z} in the Pfaff equation yields $$\tau(A)F_A(w)dz-z(1+\phi_2(z,A))dw=0,$$
where $\phi_2(z,A)$ has property (I).

The rest of the proof is as for Sxhh3.
\end{proof}

\end{document}